%% file: BDP_Arxiv_Version.tex
\documentclass[11 pt]{article}

\input{preambuleB}
\usepackage{lineno,cancel}

\def \Id{{\sf Id}}

\def \U{{\sf U}}
\def \L{{\sf L}}
\def \M{{\sf M}}
\def \DK{{\sf D}}
\def \Loop{{\sf Loop}}

\def \lt{{almost lower-triangular}}

\def \sut{{\textsf{A\small{llmost\,}\,U\small{pper}\,T\small{riangular}}}}
\def \slt{{\textsf{A\small{llmost\,}\,L\small{ower}\,T\small{riangular}}}}
\def \sut{\includegraphics[width=0.32cm]{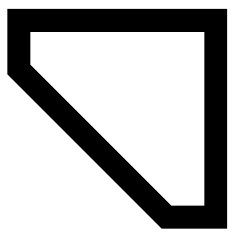}}
\def \slt{\includegraphics[width=0.32cm]{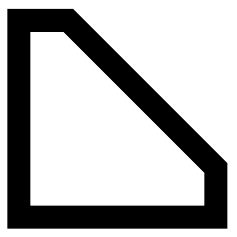}}

\def \Mp#1{{\cal M}_{#1}^{+}} 
\def \Mpq#1{{\cal M}_{#1}^{+,\sim}}

\def \DI#1#2{{\det\l( \Id-#1_{#2}\r)}}
\def \se{{\sf e}}
\def \bma{\begin{bmatrix}}
\def \ema{\end{bmatrix}}
  
\begin{document}
\begin{center}
 \LARGE\bf
Almost triangular Markov chains on $\N$\\
 {\large \bf Luis Fredes$^{\dagger}$ and Jean-Fran\c{c}ois Marckert$^{*}$}
 \rm \\
  \large{$^\dagger$Universit\'e Paris-Saclay.\\
 	$^{*}$CNRS, LaBRI, Universit\'e Bordeaux}

 \normalsize 
\end{center}
 \begin{abstract} A transition matrix $\bma\U_{i,j}\ema_{i,j\geq 0}$  on $\mathbb{N}$ is said to be almost upper triangular if $\U_{i,j}\geq 0\imp j\geq i-1$, so that the increments of the corresponding Markov chains are at least $-1$; a transition matrix $\bma \L_{i,j} \ema_{i,j\geq 0}$ is said to be almost lower triangular  if $\L_{i,j}\geq 0\imp j\leq i+1$, and then, the increments of the corresponding Markov chains are at most $+1$.
  
  In the present paper, we characterize the recurrence, positive recurrence and invariant distribution for the class of almost triangular transition matrices. The upper case appears to be the simplest in many ways, with existence and uniqueness of invariant measures, when in the lower case, existence as well as uniqueness are not guaranteed. We present the time-reversal connection between upper and lower almost triangular transition matrices, which provides classes of integrable lower triangular transition matrices.\par
  These results encompass the case of birth and death processes (BDP) that are famous Markov chains (or processes) taking their values in $\mathbb{N}$, which are simultaneously almost upper and almost lower triangular, and whose study has been initiated by Karlin \& McGregor in the 1950's. They found invariant measures, criteria for recurrence, null recurrence, among others; their approach relies on some profound connections they discovered between the theory of BDP, the spectral properties of their transition matrices, the moment problem, and the theory of orthogonal polynomials. Our approach is mainly combinatorial and uses elementary algebraic methods; it is somehow more direct and does not use the same tools.
\end{abstract}
 
 \section{Introduction} 
\label{sec:Intro}
\paragraph{Notation and conventions.} The set of non-negative integers is denoted $\N$. For two integers $a<b$, $[a,b]$ will be the set of integers $\{a,a+1,\cdots,b\}$. The word ``interval'' will be added when standard real intervals are considered. \par
A transition matrix over a finite or countable state space $S$ is a matrix $\bma\M_{i,j}\ema_{i,j \in S}$ indexed by $S$, such that the $\M_{i,j}$ are non-negative real numbers, and sums to one on each row of the matrix.\par
The set of non-negative measures over $S$ (equipped with the power set sigma field), with total mass being strictly positive, or infinite, is written ${\cal M}_S^+$. \par
A measure $\pi\in{\cal M}_S^+$ is said to be invariant by $\M$ if
\ben\label{eq:fjth}\sum_{a\in S}\pi_a\M_{a,b}=\pi_b,~~~\textrm{ for all}~~b\in S.\een
Often, we will see $\pi$ as a row matrix $\pi:=\bma \pi_x, x\in S\ema$, and \eref{eq:fjth} will be written $\pi \M=\pi$. \par

We denote by $\Mpq{S}$ the set of equivalence classes of positive measures consisting of those which are equal up to a positive factor. Since the invariance by $\M$ is a class property, we will say that a transition matrix $\M$ has a single (resp. several) invariant measure in $\Mpq{S}$, when there is a single (resp. several) class of invariant measures. \par
The adjective ``recurrent'', ``irreducible'', ``aperiodic'', and ``positive recurrent'' will qualify indifferently transition matrices and Markov chains.

For any subset $F$ of $S$, $\M_{F}$ is the matrix obtained by keeping only the lines and columns of $\M$ indexed by elements of $F$ (that is $\M_{F}=\bma\M_{i,j}\ema_{i,j \in F}$). \par
The identity matrix is denoted $\Id$, and this, whatever its size is (which will be however clear from the context). For example, we will simply write $(\Id-\M)_F=\Id-\M_F$, without adding the precision that in the left-hand side, $\Id$ has size the cardinality of $S$, and in the right-hand size, that of $F$.\par

\centerline{---------------------------}

A transition matrix $\U=\begin{bmatrix} \U_{i,j}\end{bmatrix}_{0\leq i,j\leq +\infty}$  on $\mathbb{N}$ is said to be  almost upper-triangular (\sut{}) if
\ben \U_{i,j}>0 \imp  j\geq i-1\een
and a transition matrix $\L=\begin{bmatrix} \L_{i,j}\end{bmatrix}_{0\leq i,j\leq +\infty}$ is said to be almost lower-triangular (\slt{}) if  \ben\L_{i,j}>0\imp j\leq i+1.\een
Here is some ``pictures'' explaining the iconographic notation \sut{} and \slt{} of the main objects:
\[\U:=\begin{bmatrix}
\U_{0,0} & \U_{0,1} &  \U_{0,2} & \U_{0, 3} & \U_{0, 4}  &\cdots \\
 \U_{1,0}& \U_{1,1} &  \U_{1,2} & \U_{1, 3} &  \U_{1, 4} &\cdots \\
 0      & \U_{2,1} & \U_{2,2} & \U_{2, 3} &  \U_{2, 4} &\cdots \\
 0      & 0      & \U_{3,2} & \U_{3, 3} &  \U_{3 , 4} &\cdots \\
 \vdots &  \vdots &  \vdots      & \ddots & \ddots &\cdots
\end{bmatrix},~~~~\L:=\begin{bmatrix}
\L_{0,0} & \L_{0,1} &  0       & 0        & 0        &\cdots \\
\L_{1,0} & \L_{1,1} &  \L_{1,2} & 0        & 0        &\cdots \\
\L_{2,0} & \L_{2,1} & \L_{2,2} & \L_{2, 3} & 0          &\cdots \\
\L_{3,0} & \L_{3,1} & \L_{3,2} & \L_{3, 3} &  \L_{3 , 4} &\cdots \\
 \vdots &  \vdots &  \vdots & \ddots & \ddots &\cdots
\end{bmatrix}\]

This paper aims to provide a first systematic study of Markov chains following a \sut{} or a  \slt{} transition matrix. Our results will appear to generalize birth and death (BD) processes results. These latter form the  famous model of Markov chains on $\mathbb{N}$ having tridiagonal transition matrices, often represented as 
\ben\label{eq:MM}
\TTT =
\begin{bmatrix}
  r_0  & p_0  & 0  & 0 & \cdots \\
  q_1 & r_1 & p_1  & 0 & 0 & \cdots \\
  0  & q_2 & r_2 & p_2  & 0 & \cdots \\
  0  & 0 & q_3 & r_3 & p_3  & \cdots  \\
  \vdots  & \vdots & \ddots & \ddots & \ddots & \ddots 
\end{bmatrix}
\een
where $q_{i}=\TTT_{i,i-1}, r_i=\TTT_{i,i}, p_i=\TTT_{i,i+1}$.
Tridiagonal means that $\TTT_{i,j}>0 \imp |i-j|\leq 1$, and then, this is the class of transition matrices that are  simultaneously \sut{} and \slt{}: the increments of a chain with such a transition matrix  belong to $\{-1,+1,0\}$. 

Two very influential papers published in 1956 by Karlin \& McGregor \cite{KmG,KmG2} showed that the main characteristics of BD models can be exactly computed, and this is a consequence of some particular features of the algebra that comes into play.
We review some of the results they obtained.
Consider $\TTT$ tridiagonal and $\ell$ the maximal $l$ with the following property : $\TTT_{i,i+1}$ and $\TTT_{i+1,i}$ positive for $0\leq i \leq l$  with $l \in\{0,1,\cdots,\}\cup \{+\infty\}$. If $\ell$ is finite, a Markov chain with transition matrix $\M$ is irreducible when restricted to the finite state space $\cro{0,\ell}$. This case can be studied using the finite Markov chain tools (even if BDP on compact sets are interesting from the combinatorial point of view in their own right, see  Flajolet \& Guillemin\cite{FG}).
Karlin \& McGregor results concern the irreducible case over $\N$ (case $\ell=+\infty$). In this case, the Markov chain is reversible $(\pi_{k-1} \TTT_{k-1,k}= \pi_{k}\TTT_{k,k-1}$) with respect to the measure  
\ben\label{eq:sgfqd}
\pi_0=1,\textrm{ and for }a\geq 1,~~ \pi_a=\frac{p_0\cdots p_{a-1}}{q_1\cdots q_a}=\prod_{j=1}^{a} \frac{\TTT_{j-1,j}}{\TTT_{j,j-1}},\een
so that this measure is invariant by $\TTT$. Moreover, this measure is the unique invariant measure for $\TTT$ (in $\Mpq{\N}$).
There exists an invariant probability distribution if and only if 
\ben
\sum_{k\geq 1} \prod_{j=1}^{k} \frac{\TTT_{j-1,j}}{\TTT_{j,j-1}} < +\infty,
\een
and then,  this is also a necessary and sufficient condition for positive recurrence.
A Markov chain with transition matrix $\TTT$ is recurrent if and only if
\ben\label{eq:dqsdt}
\sum_{k\geq 0}\prod_{j=1}^k \frac{\TTT_{j,j-1}}{\TTT_{j,j+1}}=+\infty.\een
There is also a connection with the theory of orthogonal polynomials (see Section \ref{sec:ortho-pol} in which the connection established by Karlin \& McGregor is discussed; we refer to Schoutens \cite{MR1761401} for more information on the connections between probability and orthogonal polynomials theories).
In their papers, Karlin \& McGregor study mainly the continuous-time version of these Markov processes, whose behaviour is similar (up to a random time change) to the discrete version preferred here. The continuous setting is important in their study since differentiation with respect to time of some quantities are considered at many steps of their study. This is not the case in our approach, and we prefer to stick with the discrete-time which makes more natural the use of combinatorial tools: we will discuss the continuous case in Section \ref{sec:CTC} only.

\subsection{Main results and contents of the paper}

The nature and behaviour of almost upper or lower triangular chains are different from those of BDP. The first difference is that if $\M$ is \sut{} (or \slt{}) but not tridiagonal, then $\M$ is not reversible with respect to any positive measure $\pi$. The reason is that for such an $\M$, there exists a pair of indices $(a,b)$ such that $\M_{a,b}>0$ and $\M_{b,a}=0$, which implies that $\pi_a \M_{a,b}=\pi_b\M_{b,a}$ is not possible.

In the paper, we will add the irreducibility hypothesis virtually everywhere: consider the strongly connected components in the graph with vertex set $V=\mathbb{N}$ and directed edge set $E=\{ (i,j), \M_{i,j}>0\}$ for $\M$ being either \sut{} or \slt{}. It is easy to see that the \sut{} (resp. \slt{}) structure imposes all strongly connected components to be intervals of $\N$, since the only down-steps are $-1$ (resp. the only upsteps are $+1$). Hence, there is at most one infinite connected component (which is in this case equivalent, up to change of origin to $[0,+\infty)$), and the Markov chain on each finite component reduces to the study of a Markov chain over a finite state space (we refer to \cite{N98,P08} for more information on Markov chain techniques). Hence, requiring irreducibility in this setting is natural. \medskip

\noindent{\bf Convention:} Unless otherwise stated, all the Markov chains will be irreducible on $\mathbb{N}$. \medskip

Let us recall a fact concerning finite irreducible Markov chains which is important to have in mind before starting the description of our results.
Let $\M=\bma \M_{a,b}\ema_{0\leq a,b\leq N}$
be a standard irreducible transition matrix over $\cro{0,N}$ for some finite $N$.
 Such a Markov chain is positive recurrent, and, by the Perron-Frobeniüs theorem (see for example \cite{M20}), $\M$ possesses a unique invariant probability distribution in $\Mp{\cro{0,N}}$ which is
\ben\label{eq:trehy}
\rho_k = {\det \l( {\sf Id}-\M^{{\sf dep}(k)}\r)}\,/\,{\alpha_N},\textrm{ for } k\in\cro{0,N}
\een
where:\\
. for a matrix $A$, notation $A^{{\sf dep}(k)}$ is the matrix $A$ deprived of its $k$-th column and line,\\
. $\alpha_N$ is the only normalising constant making of $\rho$ a probability distribution.\\

As a consequence of the matrix tree theorem (see e.g. \cite{Z85}), or as consequence of the properties of the Markov chain tree theorem (\cite{Wil96,Bro89,Al90,HLT21,LFJFM}), $\det \l( {\sf Id}-\M^{{\sf dep}(k)}\r)$ is the total weight of  the set  ${\sf SpanningTrees}^{\bullet}(r)$ of rooted spanning trees $({\sf t},r)$ of the oriented graph  $G=(V,E)$ where $V=\cro{0,N}$, and $E=\{(a,b): \M_{a,b}>0\}$. More precisely it states that each spanning tree $({\sf t},r)$ is an oriented graph $(V^{{\sf t},r},E^{{\sf t},r})$, whose edges  $e=(e_1,e_2)\in E^{{\sf t},r}$ are oriented towards the root $r$ and whose weight is defined as
\ben\label{weight}{\sf Weight}({\sf t},r)=\prod_{e \in E^{{\sf t},r}} \M_{e_1,e_2},\een
  and the matrix tree theorem in these settings writes as
\ben\label{eq:rgzf}\det \l( {\sf Id}-\M^{{\sf dep}(r)}\r)=\sum_{({\sf t},r)\in {\sf SpanningTrees}^{\bullet}(r)} {\sf Weight}({\sf t},r).\een

\subsection{Content of the paper}
\paragraph{Section \ref{sec:Mut} collects the main results concerning \sut{} transition matrices.} In \Cref{theo:InvDistU} we establish that each irreducible \sut{} transition matrix $\U$ has a single invariant measure $(\pi_a,a\geq 0)$ in $\Mpq{\N}$ where
\[\pi_a=\pi_0\, \frac{\det( \Id-\U_{[0,a-1]})}{\prod_{j=1}^{a}\U_{j,j-1}},~~a\geq 1,\]
which provides a characterization for positive recurrence ($\sum_a \pi_a<+\infty$).
A necessary and sufficient condition for recurrence is given in \Cref{theo:dfdqdq}: $\lim_{b=+\infty} \U_{1,0} \frac{\det(\Id -\U_{[2,b-1]})}{\det(\Id-\U_{[1,b-1]})} = 1$. This is done thanks to some explicit formulas for the distribution of the hitting time $\tau_{S}(Y)$ of a set $S$ by a Markov chain $(Y_i, i\geq 0)$ with transition matrix $\U$ (\Cref{theo:dfdqdq} and \Cref{pro:ht}), notably, it is established that $\P(\tau_{\{0\}}(Y)<\tau_{[b,+\infty)}(Y)~|~Y_0=1)= \U_{1,0} \frac{\det(\Id -\U_{[2,b-1]})}{\det(\Id-\U_{[1,b-1]})}$.\par
The case where absorption at 0 occurs with a certain probability is discussed in \Cref{rem:abs0}, and the probability of absorption is computed.

For each $n\geq 0$, the projection transition matrix $\U^{(n)}$ is defined by restricting $\U$ to $\cro{0,n}$ (up to some boundary details, see \eref{eq:gsdf}). Since $\U^{(n)}$ is a finite state space transition matrix, when irreducible, it possesses a unique invariant distribution $\rho^{(n)}$. In \Cref{theo:equiv}, the convergence of $\rho^{(n)}$ (after  normalisation if necessary) to the unique invariant measure $\pi$ of $\U$ in $\Mpq{\N}$ is established.  \par
Last, in \Cref{pro:qfqeh2}, it is established that  some irreducible \sut{} transition matrices $\U$ have several non-proportional right eigenvectors associated with the eigenvalue 1.

\paragraph{Section \ref{sec:ALTC} collects the main results concerning \slt{} transition matrices.} 

First of all, the class of \slt{} transition matrices appear to be more complex to study than \sut{} transition matrices. Probably, the simplest explanation is the role played more or less directly by the invariance measures $\pi$ in the study of a \slt{} transition matrix $\L$. In the \slt{} case, an invariant measure is solution to $\pi_a=\sum_{b:b\geq a-1} \pi_b\L_{b,a}$, so that it relates $\pi_{a-1}$ with an infinite number of $\pi_b$ with larger indices, while in the \sut{} case,  $\pi_a=\sum_{b:b\leq a+1} \pi_b\U_{b,a}$ allows expressing $\pi_{a+1}$ with the $\pi_b$ with smaller indices (a triangular system, easy to solve, with a unique solution).

Hence, in the \slt{} case, neither uniqueness nor existence of invariant measures are guaranteed (\Cref{theo:hetjyfd}). 
In \Cref{theo:rec}, we give a characterization of \slt{} transition matrices $\L$ that are recurrent (the condition is $\lim_{b\to +\infty} \frac{\prod_{j=1}^{b-1}\L_{j,j+1}}{\det(\Id-\L_{[1,b-1]})}=0$). This is done also by the study of the distribution of the hitting times $\tau_{\{j\}}(Y)$ of a Markov chain with transition matrix $\L$.
The case where absorption at 0 occurs with a certain probability is discussed in \Cref{rem:abs0l}, and the probability of absorption is computed.
In Theorem \ref{theo:f978sdqs}, it is shown that the invariant distribution of a \slt{} transition matrix $\L$  on the \underbar{finite} set $[0,s]$, is proportional to $(\eta_a, a\in [0,s])$ where $\eta_a = \eta_0\,\det\l(\Id-\L_{[a+1,s]}\r)~\prod_{i=1}^{a}\L_{i-1,i}$. This result allows to state \Cref{pro:ruygfs} and \Cref{pro:thezyr} which provide some conditions for the convergence of the (rescaled) invariant distribution $\eta^{(m)}$ of the projected transition matrix $\L^{(m)}$ to an invariant measure $\eta$ of $\L$. 

\paragraph{Connections between \slt{} and \sut{} transition matrices}
The time-reversal of a trajectory with jump bounded from above by 1, is a trajectory with jump bounded from below by $-1$... so that it is tempting to guess that the time-reversal of a \sut{} Markov chain is a \slt{} Markov chain, and vice-versa (under their stationary regime). It turns out that the complete picture is more complex than that because \sut{} transition matrices have a single invariant measure, when the existence and uniqueness of invariant measure are not guaranteed in the \slt{} case. Hence:\\
-- time-reversal of \sut{} transition matrices are \slt{} transition matrices, \\
-- time-reversal to \slt{} transition matrices, may exist or not, and in the case where $\L$ possesses several invariant measures, several time-reversals of $\L$ can be defined, all being \sut{} transition matrices (see  Theorems \ref{theo:conn} and \ref{theo:dqti}. Recurrence and positive recurrence of any associated time reversed transition matrices are shown to be equivalent to those of $\L$.

Since \slt{} transition matrices are in general more difficult to study that \sut{} transition matrices, finding the time-reversal $\U$ of a transition matrix $\L$ provides at once an important tool to study the behaviour of $\L$-Markov chains. We then provide some results allowing one to better understand the algebraic relation between pairs $(\L,\U)$, time-reversal of each other with respect to some measures (\Cref{pro:rytou}, and \Cref{rem:LU}).

In Section \ref{sec:cata}, we made a slight change of presentation of \slt{} transition matrices $\L$ using the so-called descent kernel
\be
\bpar{ccl}
\L_{b,a } & = & v_{b}\, \DK_{b,a}, \textrm{ for }b\geq a,\\
\L_{b,b+1}  & = & 1-v_{b},~~~~b\geq 0.\epar
\ee
Hence, $v_b$ is the probability of ``go down'' from level $b$, and $\DK_{b,a}$ is the probability to descend from $b$ to $a$, when the ``go down'' direction is chosen. This representation, of course equivalent to the initial representation of \slt{} transition matrices, provides some different formulas for the researched time-reversal transition matrix $\U$ (which involves $\DK$ too, see \Cref{pro:frsgr}).\par
 In Section \ref{sec:catal}, we will change a bit of perspective -- fix $\DK$ but let $(v_i,i\geq 0)$ be freely chosen: this will provide a way to construct many integrable \slt{}-transition matrices $\L$ (\Cref{theo:machin}). This point of view is reminiscent of ``catastrophe transition matrices''  in which the descend transition matrix is the important feature of the model. This allows us to revisit some known results of the literature (see Section \ref{sec:CK}).

In Section \ref{sec:sqshtqdf} we show that our results are equivalent to the results of Karlin \& McGregor in the tridiagonal case (our formulas use determinants when it is not the case for those of Karlin \& McGregor, so that proofs are needed). 

In Section \ref{eq:fqrfdqe}, we provide a  family of integrable \slt{} transition matrices: in words, when the columns of $\L$ are almost proportional (see Definition \ref{def:fqgrzd}), then the system which allows to compute the invariant distribution is triangular (in some sense), and then can be solved. 

In Section \ref{sec:hit}, another family of integrable models is given: these are some models of \slt{} and \sut{} that can be expressed in terms of birth-death processes decomposed between some stopping times.

In Section \ref{sec:RPMC}, a fourth list of integrable models, called the repair shop Markov chain, is revisited, and treated with our main theorems (criterion of recurrence and positive recurrence are found using new methods).

Section \ref{sec:CTC} is devoted to continuous-time counterparts of our models of \slt{} and \sut{} Markov chains.

Finally, since many proofs we give use combinatorial facts (notably matrix tree theorem and heap of pieces techniques), Section \ref{sec:Tpmt} recalls these tools.

\subsection{Tools for the proofs of the main theorems}
\label{sec:Tpmt}
\subsubsection*{About the determinants of almost triangular matrices}
\label{sec:adatm}

\begin{lem}\label{lem:etyyu}
  Let $A=\begin{bmatrix}A_{i,j} \end{bmatrix}_{0\leq i,j \leq N}$ be a \underbar{finite} \sut{} matrix  (transition matrix or not, with complex coefficients). Denote by $S^{N}$ the set of increasing integer valued sequences $s=(s_0,\cdots,s_k)$ with $1\leq k\leq N$, $s_0=-1$ and $s_k= N$. For such a sequence denote by $\ell(s)=k$ its final index. We have 
     \begin{align}\label{det0}
\det(A)=  \sum_{j=0}^{N}  A_{0,j} \l(\prod_{i=1}^{j} (-A_{i,i-1})\r)\det(A_{[j+1,N]}).
\end{align}
As a consequence, 
\begin{align}\label{det1}
	\det(A)=   \sum_{s\in S^N} \l(\prod_{j=1}^{\ell(s)} A_{s_{j-1}+1,s_j}\r)\prod_{j \in[0,N-1]\setminus s} (-A_{j+1,j})
  \end{align}
where $[0,N-1]\setminus s$ is the set obtained by removing the elements of $s$ from the set $[0,N]$.
\end{lem}
\begin{rem} If $A$ is \slt{}, then it is immediate by transposition that
  \[\det(A)= 
  \sum_{s\in S^N} \l(\prod_{j=1}^{\ell(s)} A_{s_j,s_{j-1}+1}\r)\prod_{j \in[0,N-1]\setminus s} -A_{j,j+1}.\]
  \end{rem}
\begin{proof}
  Expand the determinant along the first line: 
  $\det(A)=\sum_{j=0}^{N} (-1)^j A_{0,j} \det(A^{{\sf dep}(0,j)})$
  where  the matrix $A^{{\sf dep}(0,j)}$ is  obtained by removing line 0 and column $j$ of $A$. Now, the conclusion follows the fact that in the matrix $H:=A^{{\sf dep}(0,j)}$ the $j$ first columns have non zero entries only above the diagonal, i.e.  $(H_{a,b}>0, b<j) \imp a\geq b$. Hence when one expands the determinant, to get a non zero result, the diagonal entries of the first $j$ columns must be selected, and then they are multiplied by $\det(A_{[j+1,N]})$, this concludes \eqref{det0}. Formula \eqref{det1} is proved  by recursively applying \eqref{det0}. 
\end{proof}

\subsubsection{The matrix tree theorem and related facts}
\label{sec:MTT}
Let $(G,W)$ be a weighted oriented graph, where $G:=(V,E)$ is a graph, $V$ is the set of nodes, $E\subset V^2$ the set of edges. As usual, the (oriented) edge $e=(u,v)$ is oriented toward $v$, and its weight is $W_{u,v}$. The matrix tree theorem asserts that
\[\sum_{({\sf t},r)\in{\sf SpanningTrees}^{\bullet}(r)} \prod_{e \in E^{\sf t}} W_{e}= \det\l({\sf Laplacian}(W)^{{\sf dep}(r)}\r)\]
where each edge of the tree $({\sf t},r)$ is oriented toward the root $r$,  and where ${\sf Laplacian}(W)^{{\sf dep}(r)}$ is the Laplacian matrix of $W$, in which the $r$th line and column have been removed.
When $W_{u,v} =\M_{u,v}$ for a transition matrix $\M$, ${\sf Laplacian}(\M)^{{\sf dep}(r)}= (\Id-\M)^{{\sf dep}(r)}$ (this is equivalent to \eref{eq:rgzf}).

\begin{defi} Let ${\sf Roots}\subset V $ be a set of roots, and ${\sf Nodes}\subset V$ a set of nodes (``of other nodes'', we should say). We denote by ${\sf Forests}({\sf Nodes},{\sf Roots})$ the set of forests,  a forest being  a sequence  of rooted trees $\{({\sf t}_1,r_1),\cdots,({\sf t}_k,r_k)\}$, satisfying the following constraints:\\
-- at least one tree $(k\geq 1)$,\\
-- the set of nodes $V({\sf t}_i)$ of the ${\sf t}_i$ are disjoint, and all of them are included in  ${\sf Nodes}\cup{\sf Roots}$,\\
-- the set ${\sf Nodes}$ is spanned (the union $\cup_{i=1}^k V({\sf t}_i)\supset {\sf Nodes}$),\\
-- the set of roots $\{r_1,\cdots,r_k\}\subset {\sf Roots}$,\\
-- no outgoing edges from the elements of roots: if $(u,v)\in \cup_i E({\sf t}_i)$ then $u\not\in {\sf Roots}$.\end{defi}
For any forest ${\sf f}=\{({\sf t}_1,r_i),\cdots,({\sf t}_k,r_k)\}$, set
\[{\sf Weight}_W({\sf f})= \prod_{i=1}^k {{\sf Weight}}({\sf t}_i,r_i),\]
where ${{\sf Weight}}({\sf t}_i,r_i)$ is as in \eqref{weight} with $W$ in place of $\M$ (the edges of each tree are oriented toward their root).
\begin{pro}\label{pro:pos} Consider the graph $G=(\mathbb{N}, \{(i,j): \U_{i,j}>0\})$, weighted by the transition matrix $\U$, that is $W=\U$.
  We have
  \ben
 \sum_{ F \in {\sf Forest}([0,x-1],[x,+\infty))} {\sf Weight}_\U(F)= \det( (\Id-\U)_{[0,x-1]}).\een
  \end{pro}
  \begin{proof} This can be viewed as a consequence of the matrix tree theorem in which the set ${\sf Roots}$ is identified with one node.
    \end{proof}

\subsubsection{Heap of cycles}
\label{sec:HCT}
We recall some aspects of the theory of heaps of pieces \cite{VX,CK}, and more specifically heap of cycles, which be a useful tool to prove some of our results.

Consider $\M$ a transition matrix on a finite or infinite countable graph $G=(V,E)$, meaning that $\M=(\M_{u,v},u,v\in V)$, $\M_{u,v}\geq 0 \imp \{u,v\}\in E$, and as usual, for all $u\in V$, $\sum_{v \in V} \M_{u,v}=1$.

Attribute to each path $w=(w_0,\cdots,w_{|w|})$ on $G$, the weight
\[\W(w)=\prod_{j=1}^{|w|} \M_{w_{j-1},w_j}.\]
A path $w$ is a cycle if $w_{|w|}=w_0$ and if moreover, for all $0\leq i<j<|w|$, $w_i\neq w_j$ (a simple cycle). We extend the map $\W$ to collections of paths $C:=(w(1),\cdots,w(|C|))$ in which case we set
\begin{align}\label{WCycle}\W(C)=\prod_{j=1}^{|C|} \W(w(i)).
	\end{align}

\paragraph{Path decomposition.}

A standard result from combinatorics which has proved its importance notably in the study of loop erased random walks (see e.g. Lawler \cite{Lawler1999}, Wilson \cite{Wil96}, Marchal \cite{Marchal}), is that
\begin{lem}
  There exists a weight preserving bijective map that sends the set of paths on $G$ starting at some point $v$ onto the set of pairs $({\sf saw},{\sf hc})$ where ${\sf saw}$ is a self avoiding walk on $G$ starting at $v$, and ${\sf hc}$ is a heap of cycles with maximal pieces incident to ${\sf saw}$.
\end{lem}
(The notion of maximal pieces, if not clear, is defined above \Cref{{pro:dqd}}). See e.g. Prop. 6.3. in Viennot \cite{VX} for additional details (and a proof).
A self avoiding path is a path $w$ such that $w_i=w_j \imp i=j$. A heap of cycles, is a particular instance of the concept of heap of pieces, important combinatorial concept. We refer to Viennot\cite{VX}, Krattenthaler \cite{CK}, Cartier  \& Foata \cite{CF}, Zeilberger \cite{Z85} for details, and just recall some aspects below.\par

A heap of pieces is, \underbar{informally}, a collection of pieces, that are placed on a discrete space ($E\times \mathbb{N}$, where $E$ is a set of elements, and $\N$ is the height space). The definition uses a reflexive and symmetric relation $R$ on the set of pieces $E$. Some pieces are said to be in relation, which implies that they cannot be placed at the same height (if $p R p'$, then $(p,h)$ and $(p',h)$ cannot belong to the same heap); moreover, a piece $(p,h)$, which is then placed at height $h$, must be supported by a piece $(p',h-1)$ at height $h-1$, which is related to it (that if, if $(p,h)$ is in a heap $h$, then $h$  must contain a piece $(p',h-1)$ with $p' R p$). \medskip

There are several ways to define formally the notion of heap of pieces:\\
-- as an element of a partially commutative monoid: if this point of view is adopted, a heap is a word $w_1 \dots w_m$, where the letters $w_i$ belongs to $E$, and in which pair of non-related letters commute (Cartier \& Foata \cite{CF}),\\
-- more geometrically (Viennot \cite{VX}), in which heaps $H$ are viewed as sets of finite sets of pairs $\{(x,i): x\in E, i \in \N\}$, such that
	\begin{enumerate}
		\item If $(x,i), (y,j) \in H$ and $x \mathcal{R} y$, then $i\neq j$ (pieces in relation can not be put at the same height).
		\item If $(x,i)\in H$ and $i>0$, then there exists $(y,i-1)\in H$ with $x\mathcal{R} y$ (each piece must be supported).	\end{enumerate} 

        These points of view are equivalent (Viennot \cite{VX},  Krattenthaler \cite{CK}); each heap $h$ can also be viewed as a poset $(h,\leq)$, where:  \\
         -- in the geometric point of view, $(x,i)\leq (y,j)$ if $i\leq j$ and $xRy$ (and $\leq$ is the transitive closure of this relation),\\
      -- in the Cartier-Foata point of view, for two letters $a$ and $b$ in a word, $a\leq b$ if $aRb$ and $a$ is at the left of $b$, (and $\leq$ is the transitive closure of this relation).\par   
A piece $p$ in $h$ is said to be maximal in $h$, if $h$ does not contain any piece $p'\neq p$ with $p\leq p'$.  Each heap, as a poset, possesses some maximal pieces.\par

      A trivial heap of pieces is a heap in which all pieces are at level 0, which means that the pieces it contains are not in relation.) If one uses the  partially commutative monoid point of view, a trivial heap of pieces is a heap (a word) in which all the pieces (the letters) commute.
\begin{pro}\label{pro:dqd}[Prop.5.3 in \cite{VX}]
	Let $\mathcal{M}$ be a subset of the pieces $\mathcal{B}$.  Let $W$ be a multiplicative weight function on heaps, such that for all heap $H$ its weight $W(H)$ is the product of elementary weights $W(p)$ of the pieces $p$ it contains (the weight of a piece is independent of ``its place or height'' in the heap). Then, the total weight of the heap of pieces having their maximal pieces included  in $\mathcal{M}$ is given by
	\[
		\sum_{\stackrel{H\text{ heaps in }(\mathcal{B},\mathcal{R})}{\text{maximal piece}\subset \mathcal{M}}} W(H) = \Big( \sum_{\stackrel{T\text{ trivial}}{ \text{heap in }(\mathcal{B},\mathcal{R})}} (-1)^{|T|} W(T) \Big)^{-1}\Big( \sum_{\stackrel{T\text{ trivial}}{ \text{heap in }(\mathcal{B}\setminus \mathcal{M},\mathcal{R})}} (-1)^{|T|} W(T) \Big)
	\]
  \end{pro}
Viennot in \cite[Proposition 5.3]{VX} gave this result at the level of combinatorial objects; here, we preferred a projected version, in terms of their weights (which is what we need). (See also Theorem 4.1 in \cite{CK}).

        In heap of cycles, the pieces are cycles on a given graph $G$, and two cycles are in relation if they share a vertex. The weight of a heap of cycles, according to a transition matrix $\M$, is identified with the weight of the collection of cycles it contains.
        A heap of cycles is then trivial when all the cycles it contains are non-intersecting. Denote by $A_G$ the alternating weight of trivial cycles 
        \[A_G=\sum_{C =(C(1),\cdots,C(|C|)\in {\sf Trivial~heap~of~cycles~ on~}G} (-1)^{|C|} \prod_{j=1}^{|C|} \W(C(j)),\]
        where $W(C(j))$ is as in \eqref{WCycle}.
A simple expansion of the determinant using the cycles present on a permutation allows to get
\ben
A_G= \det\l(\Id-\M\r)=0
\een
and the reason for that is that $\M$ has 1 as an eigenvalue; hence the set of heaps of cycles on $G$ has total weight $+\infty$. What is of greater interest is the value of  $A_{G\setminus S}$, the alternating weight of trivial heap of cycles avoiding some set of vertices $S$, which is
\[A_{G\setminus S}= \det\l(\Id-\M_{G \setminus S}\r),\] 
as well as its inverse corresponding to the total weight of heaps of cycles on $G \setminus S$:
\ben\label{eq:sgr}
\sum_{H\in {\sf Heap~of~Cycles~on~}G \setminus S}\W(H)= \det\l(\Id-\M_{G \setminus S}\r)^{-1}.
\een

\section{Main theorems in the almost triangular cases}
\label{sec:Mut}

\subsection{Almost upper triangular cases}

\textbf{In the \sut{} irreducible case}, there exists a unique invariant measure:
\begin{theo} \label{theo:InvDistU} If $\U$ is an irreducible \sut{} transition matrix, then $\U$ admits a unique positive invariant measure $(\pi_a,a\geq 0)\in \Mpq{\mathbb{N}}$ , which is defined (up to a constant factor  $\pi_0>0$) by
  \[\pi_a := \pi_0\, \frac{\det( \Id-\U_{[0,a-1]})}{\prod_{j=1}^{a}\U_{j,j-1}},~~a\geq 1.\]
The transition matrix $\U$ is positive recurrent if and only if
	\ben\label{eq:rzfqhy}\sum_{a=1}^\infty \frac{\det\left( \Id-\U_{[0,a-1]}\right)}{\prod_{j=1}^{a}\U_{j,j-1}}<\infty.	\een
  \end{theo}
\begin{rem}\label{rem:shh}
  \bir
  \itr The measure $\pi$ is a positive measure on $\mathbb{N}$, this is a consequence of \Cref{pro:pos} and can be seen using \eref{eq:rgzf} too, even if the matrix $\U_{[0,a-1]}$ is not a transition matrix deprived of a line and a column (but it can be obtained as such). 
  \itr We have $\pi_a =\sum_{b\leq a+1}\pi_b \U_{b,a}$ so that there is a second algorithmic method to compute directly $(\pi_a,a\geq 0)$: fix freely a value $\pi_0>0$, and then for $a\geq 1$ use the following recursion:
  \ben \label{eq:piU}
  \pi_a :=  \pi_{a-1}\, \frac{1-\U_{a-1,a-1}}{\U_{a,a-1}}-\sum_{x\leq a-2} \pi_x\, \frac{\U_{x,a-1}}{\U_{a,a-1}}.
  \een
 The equivalence of this formula with \Cref{theo:InvDistU} is not obvious, and it is even not obvious that \eref{eq:piU} produces a positive sequence $(\pi_a,a\geq 0)$. In fact, the system $\pi_a =\sum_{b\leq a+1}\pi_b \U_{b,a}$ can be written under the form \eref{eq:piU} which is a triangular system, and then given $\pi_0$, it possesses a unique solution.
 \itr The theorem applies in the tridiagonal case even if the formula seems different from Karlin \& McGregor formula \eref{eq:sgfqd} (see Section \ref{sec:sqshtqdf} for the complete explanation).
 \eir~\\
\end{rem}  

 \begin{lem}\label{lem:fdher} For any \underbar{finite} \sut{} matrix $\U$,  any $y\geq 0$  smaller than the matrix size
\ben\label{eq:fdher}
\det( \Id-\U_{[0,y]})=\det( \Id-\U_{[0,y-1]})\, (1-\U_{y,y}) - \sum_{x\leq y-1} \det( \Id-\U_{[0,x-1]}) \, \U_{x,y} \prod_{j=x+1}^{y} \U_{j,j-1}.
\een
with the convention $\det( \Id-\U_{[0,-1]})=1$.
\end{lem}
\begin{proof}
This is an application of Lemma \ref{lem:etyyu} to the \sut{} matrix $A=\Id-U$, more exactly to the matrix obtained from $\Id-U$ by the symmetry with respect to the second diagonal. 
\end{proof}
  \begin{proof}[Proof of \Cref{theo:InvDistU}] Let us establish that $\sum_x \pi_x \U_{x,y} = \pi_y$; 
since $\U$ is \sut, this is equivalent to
\ben\label{eq:yjid}
\sum_{x\leq y+1} \det(\Id -\U_{[0,x-1]}) \U_{x,y}/{\dis\prod_{j=1}^{x}\U_{j,j-1}} &=& \det(\Id -\U_{0,y-1})/{\dis\prod_{j=1}^{y}\U_{j,j-1}}.
\een
Multiply both sides  by $\prod_{j=1}^{y+2} \U_{j,j-1}$ allows seeing that \eref{eq:yjid} is equivalent to
\ben\label{eq:yjid2}
\sum_{x\leq y+1} \det(\Id -\U_{[0,x-1]}) \U_{x,y}{\dis\prod_{j=x+1}^{y+2}\U_{j,j-1}} &=& \det(\Id -\U_{0,y-1}) {\dis\prod_{j=y+1}^{y+2}\U_{j,j-1}},
\een
which holds, by \eref{eq:fdher}.

The invariant measure defines a probability measure, and is therefore positive recurrent if and only if $\sum_{i=0}^\infty \pi_i<\infty$ which gives \eref{eq:rzfqhy}.
\end{proof}

A second proof of the Theorem will be given in Section \ref{sec:SP}.
\begin{theo}\label{theo:dfdqdq} For a Markov chain $Y=(Y_i,i\geq 0)$ with \sut{} irreducible transition matrix $\U$, denote by
  \[\tau_S(Y)= \inf \{j >0: Y_j\in S\}\]
  the hitting time of the set $S$ by $Y$. 
  Set, for any $0<x<b$,
  \[u_b(x)=\P( \tau_{\{0\}}(Y) < \tau_{[b,+\infty)}(Y) ~|~Y_0=x).\] 
  We have
  \ben\label{eq:ubx}
  u_b(x)=  \frac{\det(\Id-\U_{[x+1,b-1]})}{\det(\Id-\U_{[1,b-1]})}\prod_{j=1}^x \U_{j,j-1}\een
so that $\U$ is recurrent if and only if
\ben\label{eq:recU}\lim_{b\to +\infty} u_b(1)=\lim_{b=+\infty} \U_{1,0} \frac{\det(\Id -\U_{[2,b-1]})}{\det(\Id-\U_{[1,b-1]})} = 1. 
\een
\end{theo}
In Section \ref{sec:cr} we will see that in the tridiagonal case, this criterion reduces to Karlin \& McGregor criterion \eref{eq:dqsdt}.
\begin{proof}
Recurrence is equivalent to $u_b(1)\to 1$ when $b\to+\infty$, since $\U$ is irreducible. Consider the set of paths $P_{x\to 0, [0,b-1]}$ starting at $x$ ending at the first time it reaches 0 and staying in $[0,b-1]$. Any path $w=(w_0=x,\cdots,w_f=0)$ in this set reaches all the positions $x,x-1,\cdots,1,0$, and $0$ is hit for the first time at the end of the path; denote by
    \[\ell_y :=\ell_{y}(w)= \max\{j: w_j=y\}\]
    the last passage time of $w$ at $y$. The path $w$ can be decomposed as follows.\\
    -- In the time interval $[0,\ell_{x}]$, $w$ is a path starting and ending at $x$, staying in $[1,b-1]$,\\
    -- then there is the last step $x\to x-1$.\\
    -- In the time interval $[ 1+\ell_{x},\ell_{x-1}]$, $w$ is a path starting and ending at $x-1$, staying in $[1,x-1]$, \\
    -- then there is the step $x-1\to x-2$;\\
    more generally, in the time interval $[1+\ell_j,\ell_{j-1}]$, $w$ is a path starting and ending at $j-1$, staying in $[1,j-1]$.\\
    -- The last step is $1\to 0$.\\
	
    If one denotes by ${\sf Weight}(y,[a,b])$ the weight of cycles on $[a,b]$ with maximal pieces incident to $y$, with $y\in[a,b]$. Then, its value is by Viennot \cite[Prop. 5.3]{VX}, Krattenthaler \cite[Theorem 4.1]{CK} (and Section \ref{sec:HCT})
    \[{\sf Weight}(y,[a,b])= \frac{\det(\Id-\U_{[a,y-1]}).\det(\Id-\U_{[y+1,b]})}{\det(\Id-\U_{[a,b]})},\]
    which gives simply ${\sf Weight}(y,[a,y])= \frac{\det(\Id-\U_{[a,y-1]})}{\det(\Id-\U_{[a,y]})}$.
To get \eref{eq:ubx}, just simplify the following  telescopic product
 \be
u_b(x)&=& \Big(\prod_{j=1}^x \U_{j,j-1}\Big) \frac{\det(\Id-\U_{[1,x-1]}).\det(\Id-\U_{[x+1,b-1]})}{\det(\Id-\U_{[1,b-1]})}\prod_{j=1}^{x-1} \frac{\det(\Id-\U_{[1,j-1]})}{\det(\Id-\U_{[1,j]})}.
  \ee  
\end{proof}

\begin{pro}\label{pro:ht} For $x\in (a,b)$, set $\ra{u}_{a,\geq b}(x;z)=\E\l[ z^{\tau_{\{a\}}(Y)}\1_{\tau_{\{a\}}(Y)< \tau_{[b,+\infty)}(Y)}~|~Y_0=x\r]$ the (defective) generating function of the hitting time of $a$ under the event that $a$ is reached before $b$. We have
\begin{align}\label{hit1}
	\ra{u}_{a,\geq b}(x;z) = \frac{\det\left( \Id - z\U_{[x+1,b-1]} \right)}{\det\left( \Id - z\U_{[a+1,b-1]} \right) }\prod_{j=1}^{x} (z\U_{j,j-1}).
	\end{align}
	In particular when $U$ is recurrent $\ra{u}_{0,b}(x;z)\rightarrow \ra{u}(x;z) :=\E\l[ z^{\tau_{\{0\}}(Y)}~|~Y_0=x\r]$ as $b\to +\infty$; otherwise it converges\footnote{It converges in the sense that, for all $k$, the coefficient of $z^k$ in  $\ra{u}_{0,b}(x;z)$ converge to that of $\ra{u}(x;z)$ as $b\to+\infty$. Seen as a power series in $z$, $\ra{u}_{0,b}(x;z)$ converges uniformly on each compact included in $[0,1)$ to  $\ra{u}(x;z)$}  to
    \ben\label{eq:rauxz}\ra{u}(x,z)=\E\l[ z^{\tau_{\{0\}}(Y)}\1_{\tau_{\{0\}}(Y)<+\infty}~|~Y_0=x\r].\een
\end{pro}
\begin{proof} The proof is the same as that of Theorem \ref{theo:dfdqdq}, in which the weight $\U_{i,j}$ of a step is replaced by $\U_{i,j}z$.
	Second statement: in case of recurrence, $\P(\tau_{\{0\}}(Y)\leq \tau_{\{b\}}(Y)~|Y_0=x)\to 1$ when $b\to+\infty$.    
\end{proof}
We define now the transition matrix $\U^{(n)}$, that we will call the ``projected'' transition matrix $\U$ on $[0,n]$~:
\ben\label{eq:gsdf}\bpar{ccl}
\U^{(n)}_{i,j} &=&\U_{i,j},~~\textrm{ for }~~ i\in [0,n], j\in [0,n-1]\\
\U^{(n)}_{i,n} &=&\sum_{j\geq n} \U_{i,j}.\epar
\een
It will be often used in the sequel (as well as $\L^{(n)}$ defined in \eref{eq:gsdfL}).
\begin{rem}	Let $\ar{u}_{a,\geq b}(x;z)=\E\l[ z^{\tau_{[b,+\infty)}(Y)}\1_{\tau_{\{a\}}(Y)> \tau_{[b,+\infty)}(Y)}~|~Y_0=x\r]$ be the (defective) generating function of the hitting time of $[b,+\infty)$ under the event that $[b,+\infty)$ is reached before $a$.
	We have
	\begin{align}\label{hit2}
		\ar{u}_{a,\geq b}(x;z)= \l[ (\Id-z\U^{(b)}_{[a+1,b-1]})^{-1} z\U^{(b)}_{[a+1,b-1]\times [a+1,b]} \r]_{x,b}.
	\end{align}
where $\U^{(b)}_{[a+1,b-1]\times [a+1,b]}$ is the matrix $\U^{(b)}$ in which are kept only the entries indexed by $[a+1,b-1]\times [a+1,b]$.

To prove this formula, observe that
\ben\label{eq:srhjyr}
\hspace{-10pt}\ar{u}_{a,\geq b}(x;z)=\E\l[ z^{\tau_{[b,+\infty)}(Y)}\1_{\tau_{\{a\}}(Y)> \tau_{[b,+\infty)}(Y)}~|~Y_0=x\r]
  = \E\l[ z^{\tau_{\{b\}} (\widetilde{Y})}\1_{\tau_{\{a\}}(\widetilde{Y})> \tau_{\{b\}}(\widetilde{Y})}~|~\widetilde{Y}_0=x\r]\een
where $\widetilde{Y}$ is a $\U^{(b)}$ Markov chain, since $\widetilde{Y}$ and $Y$ can be coupled so that a jump in $[b,+\infty)$ of $Y$ (from $y\in [a,b-1]$) corresponds to a jump to $b$ for $\widetilde{Y}$ (from $y\in [a,b-1]$). 

 Since $\U^{(b)}$ is a transition matrix on a finite state space, the generating function we are looking for is $\sum_{k\geq 0} \l[ (z\U^{(b)}_{[a+1,b-1]})^{k} z\U^{(b)}_{[a+1,b-1]\times [a+1,b]}\r]_{x,b}$ where $k$ counts the number of steps of the chain $\widetilde{Y}$ avoiding $a$ and $b$, before the last step that hits $b$.
\end{rem}

\begin{rem}\label{rem:abs0} Some authors consider the case where $\sum_{k\geq 0}\U_{0,k}<1$, so that, a part of the mass disappears at each passage at 0: if one adds an additional absorbing state $\dagger$ to the state space, and set $\U_{0,\dagger}=1-\sum_{k\geq 0}\U_{0,k}$ and $\U_{\dagger,\dagger}=1$, then the absorbed mass at $\dagger$ for a $\U$-Markov chain  starting from $x$ is 
  \[A_\dagger(x)=\P(\tau_{\{\dagger\}}(X)<+\infty~|~X_0=x)\]
  and the corresponding (defective) hitting time generating function is
\[a_\dagger(x;z)=\E(z^{\tau_{\{\dagger\}}(X)}\1_{\tau_{\{\dagger\}}(X)<+\infty}~|~X_0=x).\]
Recall \eref{eq:rauxz}. We have $A_\dagger(x)=a_\dagger(x;1)$ and 
\[
	a_\dagger(x;z)= \ra{u}(x;z)z \U_{0,\dagger} \sum_{k=0}^\infty \Big(z\U_{0,0}+\sum_{y = 1}^\infty zU_{0,y}\ra{u}(y;z) \Big)^k  = \frac{z\U_{0,\dagger}\ra{u}(x;z)}{1-\left(z\U_{0,0}+\sum_{y = 1}^\infty zU_{0,y}\ra{u}(y;z) \right) }
\]
where $\ra{u}$ is defined in \Cref{pro:ht}. To show this formula, a simple decomposition is sufficient. Each path going to $\dagger$ can be decomposed as:\\
-- a trajectory that goes to 0 (whose weight is taken into account by $\ra{u}(x;z)$), \\
-- a sequence of $k$  cycles from 0 to 0 (each of them contributes $(z\U_{0,0}+\sum_{y = 1}^\infty zU_{0,y}\ar{u}(y,z))$,\\
-- and then a step leading to $\dagger$ from 0.
\end{rem}

\begin{defi}
Let $\pi,\pi^{(1)},\pi^{(2)},\cdots,$  be a sequence of measures on $\mathbb{N}$. The sequence  $(\pi^{(n)})$ is said to converge weakly to $\pi$ if for all $k\geq 0$, $\lim_n \pi_k^{(n)}=\pi_k$.
\end{defi}
When these measures are probability measures on $\mathbb{N}$, this is the classical convergence in distribution.
\begin{theo}\label{theo:equiv} Let $\U$ be a \sut{} transition matrix, irreducible on $\mathbb{N}$, and $\U^{(n)}$ be the projected transition matrix defined in \eref{eq:gsdf}.
Denote by $\rho^{(n)}$ the unique invariant probability distribution of $\U^{(n)}$. 
\bir
\itr The transition matrix $\U$ admits $\pi \in\Mp{\mathbb{N}}$ as an invariant measure, if and only if there exists a sequence  $(c_n, n\geq 0)$ such that $c_n \rho^{(n)} \to \pi$  weakly.
\itr $\U$ is positive recurrent with invariant probability distribution $\rho$ iff $\rho^{(n)} \to \rho$  weakly.
\eir
\end{theo}
\begin{proof}Since $\U_{i,j}=\U_{i,j}^{(n)}$ for $i<n$, the equilibrium equations
  \be
  \pi_b &=& \sum_{a\leq b+1} \pi_a \U_{a,b},\\
  \rho^{(n)}_b&=&\sum_{a\leq b+1} \rho^{(n)}_{a} \U^{(n)}_{a,b}=\sum_{a\leq b+1} \rho^{(n)}_{a} \U_{a,b}
  \ee
  are the same for $b\leq n-1$. These systems can be rewritten to express $\pi_{b+1}$ (respectively $\rho^{(n)}_{b+1}$) in terms of  $\pi_j$ with smaller indices $j$, as follows:
  \ben\label{eq:tryrs}
  \pi_{b+1} &=& (\pi_{b}- \sum_{a\leq b} \pi_a \U_{a,b}) /U_{b+1,b}, \\
  \rho^{(n)}_{b+1} &=& (\rho^{(n)}_{b}- \sum_{a\leq b} \rho^{(n)}_{a} \U_{a,b}) /U_{b+1,b}
  \een
  for $b\leq n-1$. Fixing a value for $\pi_0$, this allows deducing the proportionality 
  \ben\label{eq:hgzf}
  \l(\pi_i, 0\leq i \leq n-1\r) = C_n \big(\rho^{(n)}_i, 0\leq i \leq n-1\big)\een
  for a constant $C_n>0$. The uncontrolled weight $\rho^{(n)}_n$ is not a detail at all, since it is directly related to $C_n$. If $C_n$ goes to $+\infty$, for example, it means that the mass $\rho^{(n)}_i$ vanishes when $n\to +\infty$, but this does not prevent $C_n\rho_i^{(n)}$ to converge.\\
\underbar{Proof of $(i)$.} Assume that $\pi$ is invariant by $\U$, by  \eref{eq:hgzf}, $C_n \rho^{(n)}\to \pi$ weakly. Conversely, assume that $C_n\rho^{(n)}\to \pi$. Still by \eref{eq:hgzf} and \eref{eq:tryrs}, $\pi$ is invariant by $\U$.\\
\underbar{Proof of $(ii)$.} First, if $\rho^{(n)}\to \rho$ weakly, then, since $\rho$ is assumed to be a probability distribution, by $(i)$ $\rho$ is invariant by $\U$, and since $\rho$ is summable and $\U$ irreducible, then $\U$ is positive recurrent.\par
  Conversely, assume that $\rho$ is invariant by $\U$. 
  We define on the same probability space $X=(X_i, i\geq 0)$ and $X^{(n)}=(X_i^{(n)}, i\geq 0)$, such that $X$ is a $\U$-Markov chain, and $X^{(n)}$ is a $\U^{(n)}$-Markov chain as follows: take $X_0=X^{(n)}_0$ with probability 1 as initial distribution (for all $n\geq 0$).  For any $n\geq 1$, consider the random subset $S^{(n)}= \{k: X_k\leq n\}$ and sort its elements increasingly $\tau_{0}^{(n)},\tau_1^{(n)},\cdots$. It is easy to check that $\big(X_{\tau^{(n)}_i}, i \geq 0\big)$ is a $\U^{(n)}$-Markov chain. From here, since $\U$ is positive recurrent, the ergodic theorem applies, and $|\{i: X_i=x, i \in [0,m-1]\}|/m\rightarrow \rho_x$ as $m\rightarrow \infty$ (a.s.), and $\sum_{N\geq n}\rho_N\to 0$ (a.s.) when $n\to+\infty$ (since $\rho$ is tight, as a probability measure). Applying the ergodic theorem to $X^{(n)}$ too, allows seeing that for $\rho^{(n)}_k = \rho_k$ for $k\leq n-1$ and $\rho^{(n)}_n= \sum_{N\geq n}\rho_N$. Hence $\rho^{(n)}\rightarrow \rho$, weakly.
\end{proof}

\begin{pro}\label{pro:qfqeh2} There exists some \sut{} irreducible transition matrices $\U$ with several positive, non-proportional right eigenvectors associated with the eigenvalue 1.
\end{pro}
We postpone the proof to \Cref{sec:teheesgs}, since it is related to the existence of \slt{} transition matrices with several invariant measures.

\subsubsection{A second proof of \Cref{theo:InvDistU}}
\label{sec:SP}Consider $\U$ an irreducible \sut{} transition matrix.
Since the transition matrix $\U^{(n)}$ is finite, we have
  \ben\label{eq:mttu1}
  \rho^{(n)}_a= \alpha_n \det\l(\Id-\U^{(n)}{}^{{\sf dep}(a)}\r).
  \een
  Here $\alpha_n$ is the only constant making of $\rho^{(n)}$ a probability distribution.~\\
   Now, we claim that for some constants $\alpha_n',\alpha''_n$, for all $a \in [0,n-1]$,
  \ben
 \det\l(\Id-\U^{(n)}{}^{{\sf dep}(a)}\r)
\label{eq:tjt2}  &=&\alpha'_n\,\DI{\U}{[0,a-1]}\prod_{j=a+1}^{n}\U_{a,a-1}\\
\label{eq:tjt3}  &=&\alpha''_n\,\DI{\U}{[0,a-1]}/\prod_{j=1}^{a}\U_{a,a-1}.
  \een
The invariant distribution $\pi=[\pi_a, 0\leq a\leq n]$ of any irreducible transition matrix $\M$ indexed by $[0,n]$ is proportional to $\det({\Id -\M}^{{\sf dep}(a)})$ (Section \ref{sec:MTT}). If such $\M$ is a
 \sut{} transition matrix, then each tree rooted at $a\in\cro{0,n}$ can be decomposed in two parts: a branch $n\mapsto n-1 \mapsto \cdots \mapsto a$ ``above $a$'', and a forest with set of roots on $[a,n]$, and other vertices on $[0,a-1]$, so that, as explained in Proposition \ref{pro:pos} leads to \eref{eq:tjt2} (since $\U^{(n)}$ is a \sut{}-transition matrix on a finite state space, and since $\prod_{j=a+1}^{n}\U^{(n)}_{a,a-1}=\prod_{j=a+1}^{n}\U_{a,a-1}$).
  Formula \eref{eq:tjt3} is obtained by dividing \eref{eq:tjt2} by the constant (depending only on $n$) $\prod_{j=1}^{n}\U^{(n)}_{a,a-1}=\prod_{j=1}^{n}\U_{a,a-1}$.

  Using \eref{eq:mttu1}, \eref{eq:tjt2} and \eref{eq:tjt3}, one sees that
  \[\frac{\rho^{(n)}_a}{\alpha_n\alpha''_n}= \DI{\U}{[0,a-1]}/\prod_{j=1}^{a}\U_{a,a-1}\]
  so that this constant sequence converges when $n\to+\infty$. Hence, Theorem \ref{theo:equiv}$(i)$ applies:  the measure $(\DI{\U}{[0,a-1]}/\prod_{j=1}^{a}\U_{a,a-1},a\geq 0)$ is invariant by $\U$.

\subsection{Almost lower triangular cases}
\label{sec:ALTC}
In the \slt{} case, neither uniqueness nor existence of invariant measures are guaranteed:
 \begin{theo}\label{theo:hetjyfd} Let $\L$ be an irreducible \slt{} transition matrix.
     \bir
     \itr $\L$ has a unique right eigenvector associated with the eigenvalue 1 (up to a multiplicative constant), and this is the vector whose entries are all equal to one.
     \itr The three following cases arise: $(a)$ $\L$ has no invariant measure in $\Mpq{\mathbb{N}}$, $(b)$ $\L$ has a unique invariant measure in $\Mpq{\mathbb{N}}$, $(c)$ $\L$ has several invariant measures in $\Mpq{\mathbb{N}}$.
     \eir
   \end{theo}
   Proof of $(ii)(c)$ is postponed to \Cref{sec:teheesgs}.
   \begin{proof}[Proof of $(i)$ and $(ii)(a)$ and $(b)$]
     $(i)$ Write the system $\L R=R$ under the  triangular form $R_{i+1}=(R_i-\sum_{j:j\leq i} \L_{i,j}R_j)/\L_{i+1,i}$, for $i\geq 0$, so that the choice of $R_0=1$ fixes all the other entries to 1.\\
     $(ii)(b)$ All tridiagonal irreducible transition matrices have a unique invariant measure in $\Mpq{\mathbb{N}}$ since they are \sut{}, and then Theorem \ref{theo:InvDistU} applies.\\
     $(ii)(a)$ In some lecture notes on Markov chains, as an example of Markov chain on $\N$ with no invariant distribution, a transition matrix of type $\slt{}$ is often given (see e.g. \cite[Example 1.7.11]{N98}). For sake of completeness, we provide a similar example here.
     Take $\L_{0,1}=a_0=1$, and for $i\geq 1$, $\L_{i,i+1}=a_i$, $\L_{i,0}=1-a_i$. Notice that if $a_i\in (0,1)$ for all $i\geq 0$, then $L$ is irreducible. An invariant measure $\pi$ would satisfy  for $i>0$, $\pi_{i}=\pi_{i-1}\L_{i-1,i}$, so that $\pi_i=\pi_0\prod_{j=0}^{i-1}a_j$ and therefore $\pi_0= \sum_{k\geq 1}\pi_k \L_{k,0}=\pi_0\sum_{k\geq 1} (1-a_k)\prod_{j=0}^{k-1}a_j$. From this we see that: $\L$ has an invariant measure if and only if $\sum_{k\geq 1} (1-a_k)\prod_{j=0}^{k-1}a_j=1$. 
     Consider $a_0=1$, and for $j\geq 1$, $a_j= 1 -\frac{1}{(j+1)^2}=\frac{j(j+2)}{(j+1)^2}$  and then, since it is a telescopic product
     \[\sum_{k\geq 1} (1-a_k)\prod_{j=1}^{k-1}a_j= \sum_{k\geq 1}  (1-a_k)\prod_{j=1}^{k-1}\frac{j(j+2)}{(j+1)^2} =\sum_{k\geq 1}\frac{1}{(k+1)^2}  \frac{k+1}{2k}=1/2.\]
\end{proof}

\begin{theo}\label{theo:rec}  Let $\L$ be a \slt{} irreducible transition matrix and $(Y_i,i\geq 0)$ a $\L$-Markov chain. Set
  \[v_b(x)=\P(\tau_{\{0\}} (Y)>\tau_{\{b\}}(Y)~|~Y_0=x), ~~\textrm{ for } 0<x<b.\]
  We have
  \ben\label{eq:vbx}
  v_b(x)= \frac{\det(\Id-\L_{[1,x-1]})}{\det(\Id-\L_{[1,b-1]})}\prod_{j=x}^{b-1}\L_{j,j+1},\een
  and then, the transition matrix $\L$ is recurrent if and only if
  \ben\label{eq:srhs}
  \lim_{b\to +\infty} v_b(1)=\lim_{b\to+\infty}\frac{\prod_{j=1}^{b-1}\L_{j,j+1}}{\det(\Id-\L_{[1,b-1]})}=0.
\een 
\end{theo}

\begin{rem}\bir
  \itr Observe that $v_b(x)$ is not defined as $u_b(x)$ (of \Cref{theo:dfdqdq}), but rather, corresponds to $1-u_b(x)$. Note that since $Y$ has its jumps bounded by $+1$, it hits $[b,+\infty)$ at $b$ (starting from $x$). See Remark \ref{rem:refsq} for a subtle point. 
  \itr By the forthcoming \Cref{pro:rytou}, Condition \eref{eq:srhs} in \Cref{theo:rec} is equivalent to 
  \ben\frac{ \det( \Id-\L_{[0,b-1]}) }{\det(\Id-\L_{[1,b-1]})}\to 0.\een
  \eir \end{rem}
\begin{proof}[Proof of Theorem \ref{theo:rec}] The proof is similar to that of \Cref{theo:dfdqdq}.
Recurrence is equivalent to $v_1(b)\to 0$ when $b\to+\infty$.
Here, for $0<x<b$, use the last hitting time of the numbers in $[x,b-1]$ in the decomposition and get \eref{eq:vbx} by simplifying the following telescopic product
\be
v_b(x) & = & \Big(\prod_{j=x}^{b-1}\L_{j,j+1}\Big)\frac{\det(\Id-\L_{[1,x-1]})\det(\Id-\L_{[x+1,b-1]})}{\det(\Id-\L_{[1,b-1]})}\prod_{j=x+1}^{b-1} \frac{\det(\Id-\L_{[j+1,b-1]})}{\det(\Id-\L_{[j,b-1]})}.
\ee
\end{proof}

An analogue of \Cref{pro:ht}:
\begin{pro}\label{pro:fqqfd} For $x\in (a,b)$, set $\ar{v}_{\leq a,b}(x;z)=\E\l[ z^{\tau_{\{b\}}(Y)}\1_{\tau_{[0,a]}(Y)> \tau_{\{b\}}(Y)}~|~Y_0=x\r]$ and $\ra{v}_{\leq a,b}(x;z)=\E\l[ z^{\tau_{\{a\}}(Y)}\1_{\tau_{[0,a]}(Y)< \tau_{\{b\}}(Y)}~|~Y_0=x\r]$.
We have
\be
\ar{v}_{\leq a,b}(x;z) &=& \frac{\det\left( \Id - z\L_{[a+1,x-1]} \right)}{\det\left( \Id - z\L_{[a+1,b-1]} \right) }\prod_{j=x}^{b-1} (z\L_{j,j+1})\\
\ra{v}_{\leq a,b}(x;z)&=& \sum_{j=0}^a\l[ (\Id -z\L_{[a+1,b-1]})^{-1} (z\L_{[a+1,b-1]\times [0,b-1]})\r]_{x,j}.
\ee
  \end{pro}
	The proof is a simple adaptation of that of \Cref{pro:ht}.
\begin{rem}\label{rem:abs0l} Absorption at 0: consider a \slt{} transition matrix $\L$, for which $\sum_{k\geq 0}\L_{0,k}<1$, and add  again an additional absorbing state $\dagger$ to the state space, and set $\L_{0,\dagger}=1-\sum_{k\geq 0}\L_{0,k}$ and $\L_{\dagger,\dagger}=1$. The absorbed mass at $\dagger$ starting from $x$ is 
  $B_\dagger(x)=\P(\tau_{\{\dagger\}}(Y)<+\infty~|~Y_0=x)$
  for $\L$-Markov chain $Y$,
and the corresponding (defective) hitting time generating function is
$b_\dagger(x;z)=\E(z^{\tau_{\{\dagger\}}(Y)}\1_{\tau_{\{\dagger\}}(Y)<+\infty}~|~Y_0=x)$.
Recall \eref{eq:rauxz}. We have $B_\dagger(x)=b_\dagger(x;1)$ and 
\[b_\dagger(x;z)= \frac{z\L_{0,\dagger}\ra{v}(x;z)}{1-\left(z\L_{0,0}+\sum_{y = 1}^\infty zL_{0,y}\ra{v}(y;z) \right) }
\]
where $\ra{v}(x;z)=\lim_{b\to +\infty}\ra{v}_{\leq 0,b}(x;z)$.\end{rem}

    An invariant measure $\eta$ satisfies $\eta_b= \sum_{a\geq b-1} \eta_a \L_{a,b}$ so that
\ben\label{eq:sgtj}
\eta_{b-1}=\frac{\eta_b(1-\L_{b,b})-\sum_{a\geq b+1} \eta_a\L_{a,b}}{\L_{b-1,b}}. 
\een
The fact that $\eta_{b-1}$ is expressed using the $\eta_a$ with larger indices $a$ brings a very important difficulty here: formula \eref{eq:sgtj} can be used to check that a sequence $(\eta_k, k\geq 0)$ is indeed invariant, but, it seems unsuitable to compute an invariant distribution; and once again, such a solution does not exist in all generality. 
\begin{theo}\label{theo:f978sdqs}  Let $\L$  be an irreducible \slt{} transition matrix with \underbar{finite size} (indexed by $[0,s]\times[0,s]$). For any $\eta_0>0$, set for $a\in[0,s]$,
  \ben\label{eq:fqq8g} \eta_a = \eta_0 \,\det\l(\Id-\L_{[a+1,s]}\r)~\prod_{i=1}^{a}\L_{i-1,i}.\een
 The measure $(\eta_a,a\in[0,s])$ is invariant by $\L$ (and by Perron-Frobeniüs, there is a single class of invariant measures).
 \end{theo}
 \begin{proof} This is a consequence of \eref{eq:rgzf} and of \Cref{pro:pos}.
Indeed, observe the geometry of the graph with vertex set $[0,s]$ and edge set $\{(i,j): \L_{i,j}>0\}$.  Take any spanning tree rooted at $a$: The vertices in $[0,a-1]$ can be connected to $a$ only using the edges $0\mapsto 1 \mapsto \cdots \mapsto a$ (so that the observed tree contains this branch), and the rest of the edges of the tree, forms a forest whose root set is contained in $[0,a]$ having set of nodes $[a+1,s]$.
   \end{proof}
Given this theorem, it is tempting to think that when $\L$ is indexed by $\N$, and say, irreducible, its invariant distribution is obtained by just taking $\lim_n \det\l(\Id-\L_{[a+1,n]}^{(n)}\r)~\prod_{i=1}^{a}\L_{i-1,i}$
where $\L^{(n)}$ is the projected transition matrix of $\L$ on $[0,n]$ defined by
\ben\label{eq:gsdfL}\bpar{ccl}
\L^{(n)}_{i,j} &=&\L_{i,j},~~\textrm{ for }~~ 0\leq i \leq n, 0\leq j \leq n-1\\
\L^{(n)}_{i,n} &=&\sum_{j\geq n} \L_{i,j}.\epar
\een
But it is not the case, since Theorem \ref{theo:hetjyfd} establishes that a transition matrix $\L$ is not assured to have an invariant distribution. The complete picture is more complex and some additional conditions are needed to get this kind of convergence result:
\begin{pro}\label{pro:ruygfs} Let $\L$ be a \slt{} irreducible transition matrix. 
Let $\rho^{(n)}$ be the invariant probability distribution of $\L^{(n)}$ (see \eref{eq:fqq8g}).
Set\[\eta^{(n)}_a:= {\rho^{(n)}_a}\,/\,{\rho^{(n)}_0},~~\textrm{ for } a\geq 0.\]
If the three following conditions hold:\\
(a)  there exists a non-negative sequence $(S_a,a\geq 0)$ such that, for each $b$, $\sum_{a:a\geq b+1}S_a\L_{a,b}<+\infty$, and which bounds uniformly $\eta^{(n)}$~: for all $a,n\geq 0$, $|\eta^{(n)}_a|\leq S_a$,\\
(b) $\lim_n\eta^{(n)}_a $ exists for each $a$; set $\eta_a:=\lim_n\eta^{(n)}_a$ for $n\to +\infty$,\\
(c) for every $a$, $\eta^{(n)}_n \sum_{j\geq n}\L_{j,a}\to 0$\\
then $\eta$ is invariant by $\L$.
\end{pro}

\begin{proof} 
	We will prove that $\eta$ satisfies \eqref{eq:sgtj}
Fix some $b\in\mathbb{N}$, and take $n>b$. Since $\rho^{(n)}=\rho^{(n)}\L^{(n)}$, 
\ben\label{eq:fhrhrs}
\rho^{(n)}_{b-1}= (1-\L_{b,b})\rho^{(n)}_b +\sum_{a=b+1}^{n-1}\rho^{(n)}_a \L_{a,b}+\rho^{(n)}_n \sum_{a\geq n}\L_{a,b}\een
or equivalently
\[\eta^{(n)}_{b-1}= (1-\L_{b,b})\eta^{(n)}_b +\sum_{a=b+1}^{n-1}\eta^{(n)}_a \L_{a,b}+\eta^{(n)}_n \sum_{a\geq n}\L_{a,b}.\]
By $(c)$,  $\eta^{(n)}_n \sum_{a\geq n}\L_{a,b}\to 0$, and then by Lebesgue dominated convergence theorem (using $(a)$ and $(b)$), $\sum_{a=b+1}^{n-1}\eta^{(n)}_a \L_{a,b}\to \sum_{a:a\geq b+1}\eta_a \L_{a,b}$. Finally, since $(\eta^{(n)}_{b-1}, (1-\L_{b,b})\eta^{(n)}_b)\to (\eta_{b-1}, (1-\L_{b,b})\eta_b)$ when $n\to +\infty$, the conclusion follows.
\end{proof}
\begin{pro}\label{pro:thezyr}Let $\rho^{(n)}$ be the invariant probability distribution of $\L^{(n)}$.
If $\rho^{(n)}$ converges weakly to some probability measure $\rho$ on $\N$, then $\rho$ is invariant by $\L$.
\end{pro}
\begin{proof} Consider \eref{eq:fhrhrs} which is equivalent to $\rho^{(n)}\L^{(n)}=\rho^{(n)}$. 
Since $(\rho^{(n)}_{b-1}, (1-\L_{b,b})\rho^{(n)}_b)\to (\rho_{b-1}, (1-\L_{b,b})\rho_b)$, to conclude that $\rho \L=\rho$, it suffices to establish that for every $b\in\mathbb{N}$,
\ben
\sum_{a=b+1}^{n-1}\rho^{(n)}_a \L_{a,b}\sous{\longrightarrow}{n\to+\infty} \sum_{a= b+1}^{+\infty}\rho_a \L_{a,b}.\een
Take a small $`e>0$. As a measure over $\mathbb{N}$, $\rho$ is tight: there exists $K$ such that $\rho_0+\cdots+\rho_K>1-`e$. Take now $n$ large enough, so that
 $\rho_0^{(n)}+\cdots+\rho_K^{(n)}>1-2`e$, so that $\sum_{j>K}\rho_j^{(n)}\leq 2`e$.
Since $\L_{a,b}\leq 1$, for all $b$,
 \[\l|\sum_{a=b+1}^{n-1}\rho^{(n)}_a \L_{a,b}- \sum_{a= b+1}^{+\infty}\rho_a \L_{a,b}\r|\leq \l|\sum_{a=b+1}^{K}\rho^{(n)}_a \L_{a,b}- \sum_{a= b+1}^{K}\rho_a \L_{a,b}\r|+2`e\]
 (with the empty sum being equal to 0, when $b+1>K$).
Now, since $\rho^{(n)}\to\rho$ weakly, the r.h.s. is smaller than $3`e$ for $n$ large enough.
\end{proof}

\begin{rem} Tridiagonal transition matrices are \slt{}; some work is needed to see that the results of this section applies to the tridiagonal case (see Section \ref{sec:sqshtqdf}).
  \end{rem}

\section{Connections between  almost upper and lower triangular cases}
\label{sec:Connect_Up_Low}

According to \Cref{theo:InvDistU}, \sut{} transition matrices always have an invariant measure, while it is not the case for \slt{} cases (\Cref{theo:hetjyfd}). The next theorem says that one can associate with each \sut{} transition matrix a \slt{} one (its time-reversal).
\begin{theo}\label{theo:conn}
Consider an irreducible \sut{} transition matrix $\U=\begin{bmatrix} \U_{i,j}\end{bmatrix}_{0\leq i,j}$, with invariant measure $\pi$, then set $\L=\begin{bmatrix} \L_{i,j}\end{bmatrix}_{0\leq i,j}$ as
\ben\label{eq:tgdq}
\L_{i,j}= \pi_j\U_{j,i} /\pi_i. \een
\bir
 \itr $\L$ is an irreducible \slt{} transition matrix on $\mathbb{N}$, with invariant measure $\pi$ too.

\itr $\L$ is recurrent if and only if $\U$ is recurrent, 
\itr If  $\pi$ is a probability distribution then, if $(Y_k,k\in \Z)$ is a $\U$-Markov chain under its stationary regime (meaning that $Y_k\sim \pi$ for any $k\in \mathbb{Z}$), then the time-reversal of this chain, $(Y_{-k},k \in \Z)$ is a $\L$-Markov chain  under its stationary regime.
\itr $\L$ is positive recurrent if and only if $\U$ is positive recurrent.
\eir
 \end{theo}
 \begin{proof}
   $(i)$: straightforward.\\
   $(ii)$: for an irreducible $\U$-Markov chain $(Y_j,j\geq 0)$, recurrence is equivalent to $\P(\tau_{\{0\}}(Y)<+\infty~|~Y_0=0)=1$. This means that the total ``$\U$-weights'' of the paths in the set $\cup_{k\geq 0}\{(x_0=0,x_1,\cdots,x_k,x_{k+1}=0), x_i>0, i\in[1,k]\}$ is 1 when the ``$\U$-weight'' of a given path $(x_0,\cdots,x_{k+1})$ is defined to be $\prod_{j=0}^k \U_{x_j,x_{j+1}}$. Since such paths start and end at 0, then their $\L$ weights and $\U$-weights coincide.
\\  $(iii)$: by translation invariance, it suffices to write
   \[\P(Y_k=y_k, 0\leq k \leq a)= \pi_{y_0} \prod_{j=0}^{a-1}\U_{y_j,y_{j+1}}= \pi_{y_a}\prod_{j=0}^{a-1} \L_{y_{j+1},y_j}=\P(Z_{a-j}=y_j, 0\leq j\leq a)\]
for $Z$ a $\L$-Markov chain under its invariant regime.\\
$(iv)$: by $(i)$, both $\L$ and $\U$ have the same invariant measure (which implies the statement).\\
\end{proof}
As a consequence of the Theorem we have 
\ben\label{eq:grsgfq}
\U_{1,0} \frac{\det(\Id -\U_{[2,b-1]})}{\det(\Id-\U_{[1,b-1]})}\xrightarrow[b\to \infty]{} 1 \equi \frac{\prod_{j=1}^{b-1}\L_{j,j+1}}{\det(\Id-\L_{[1,b-1]})}=\frac{ \det( \Id-\L_{[0,b-1]}) }{\det(\Id-\L_{[1,b-1]})}\xrightarrow[b\to \infty]{} 0,\een
but the value of the left-hand side of these formula are different, in general, for any fixed $b$.
See Remark \ref{rem:refsq} to explore further ``what is equal''.

\begin{theo}\label{theo:dqti} The \slt{} transition matrix $\L$ admits a time-reversal transition matrix $\U$   if and only if it possesses a positive invariant measure $\eta$ in which case $\U_{b,a}=\eta_a \L_{a,b} / \eta_b$, and $\U$ and $\L$ are both time-reversal of each other. As a consequence, for each $\L$, there is a bijection between the set of classes of invariant measures of $\L$ (in $\Mpq{\N}$) and the set of time-reversal transition matrices $\U$.
  \end{theo}
The proof of this theorem is simple since any such $\U$ is the time-reversal of $\L$, but it exists only when the positive invariant measure $\eta$ exists; as explained in \Cref{theo:hetjyfd}, some irreducible \slt{} do not admit any positive invariant measure.

\subsection{Algebraic connection between $\U$ and $\L$}

Consider $(\U,\L)$ a pair of irreducible transition matrices where $\U$ is \sut, $\L$ is \slt, and assume that they are time-reversal of each other. 
The invariant measure $\pi$ of $\U$ is unique, so that $\L_{i,j}=\pi_j\U_{j,i}/\pi_i$, $\forall i,j\geq 0$. 
A simple expansion of the determinant using the cycles decomposition of permutations, give, for every $a,b\geq 0$: 
  \ben\label{eq:tjt}
   \det\l( \Id-\U_{[a,b]}\r)=\det\l( \Id-\L_{[a,b]}\r).
  \een
Apart this formula, the main relation is
\ben\label{eq:base}
\L_{a,b}=\pi_{b}\U_{b,a}/\pi_a &=& \U_{b,a} \frac{\det\l( \Id-\U_{[0,b-1]}\r)/\prod_{j=1}^{b}\U_{j,j-1}}{  \det\l( \Id-\U_{[0,a-1]}\r)/\prod_{j=1}^{a}\U_{j,j-1}}.
\een
If $\U$ is known, and the corresponding $\L$ is searched, then this last formula, built using Theorem \ref{theo:InvDistU} allows to compute it. On the other hand, if $\L$ is known, but not $\U$, this is more difficult since we have no simple expression of $\pi$ in terms of $\L$ (and again, the existence and uniqueness of $\pi$ are not assured).

The following proposition provides some relations between the elements in the tuple $(\pi,\U,\L)$.
\begin{pro}\label{pro:rytou}For any $b\geq 0$,  set
  \[Z_b:=\prod_{j=1}^b \frac{\L_{j-1,j}}{\U_{j,j-1}},~~ Z'_b:=\prod_{j=1}^b \frac{\U_{j-1,j}}{\L_{j,j-1}},\] (where $Z_0=Z'_0=1$, which is compatible with the convention concerning empty products),
  \bir 
  \itr For any $a\geq 0$, $\dis{\det( \Id-\L_{[0,a-1]})}={\dis  \prod_{j=0}^{a-1} \L_{j,j+1}}$ (with the convention, $\det( \Id-\L_{[0,-1]})=1$).
  \itr For any  $b\geq 0$, $Z_b=Z'_b$.
  \itr The measure $(Z_0,Z_1,Z_2,\cdots)$ is invariant by both $\U$ and $\L$.
  \eir 
\end{pro}
The point $(ii)$ of \Cref{pro:rytou} is equivalent to 
\ben \label{eq:sfdsd}
\L_{a,a-1}\L_{a-1,a} =\U_{a-1,a} \U_{a,a-1}.\een 
\begin{proof}
  Taking $b=a+1$ in \eref{eq:base}, gives
\ben\label{eq:yrjugjg1}
\L_{a,a+1}= {\det\l( \Id-\U_{[0,a]}\r)}/{  \det\l( \Id-\U_{[0,a-1]}\r)}\een
from what we infer $(i)$ (using \eref{eq:tjt}).\par
For $k\geq 1$, using \eref{eq:base} we get
\ben\label{eq:htegfe}
\L_{b+k,b} \prod_{j=b+1}^{b+k} \L_{j-1,j}
=\U_{b,b+k}\prod_{j=b+1}^{b+k}\U_{j,j-1};
\een
for $k=1$ this provides $(ii)$. Further this equation rewrites $\L_{b+k,b}=Z_b \,\U_{b,b+k}\,/\, Z_{b+k}$, and since for $k=-1$ this is valid too 
(since $\L_{b-1,b}=Z_b \U_{b,b-1}/Z_{b-1}=U_{b,b-1}U_{b-1,b}/\L_{b,b-1}$, and this is true by $(ii)$) we have for all $a,b$, $\L_{a,b}=Z_b \,\U_{b,a}\,/\, Z_{a}$, which ensures $(iii)$.
\end{proof}

\begin{rem}\label{rem:LU}By \Cref{pro:rytou}, when $\L$ is known, each of the sequences $(\pi_a,a\geq 0)$ (up to a multiplicative constant), $(\U_{a,a-1},\;a\geq 1)$,$(\U_{a,a+1},\:a\geq 0)$, $((\U_{i,j},\;j>i),i\geq 0)$ allows to compute the others. For example, if $(\U_{a,a+1},a\geq 0)$ is known as well as $\L$, \eref{eq:sfdsd} allows to compute $(\U_{a,a-1}, a\geq 1)$, then \eref{eq:htegfe} allows to compute $((\U_{i,j}),j>i,i\geq 0)$, and $(Z_b,b\geq 0)$ which is proportional to $\pi$.  
 
\end{rem}

\begin{rem}\label{rem:refsq}
By reversibility, it can be seen that $u_b(x)$ and $1-v_b(x)$ are not equal in general. This comes from a lack of symmetry in the measured event. If instead ones observe the return time to 0 by random walks starting at 0, the symmetry comes back; but the formula are more complex. Denote by $Y^\U$ and $Y^\L$ Markov chains with respective transition matrices $\U$ and $\L$ such that $\pi_{a} \L_{a,b}=\pi_b\U_{b,a}$. 
  One has
  \ben \label{eq:retf}\P(\tau_{\{0\}}(Y^\U)< \tau_{[b,+\infty)}(Y^\U)~|~Y^{\U}_0=0)=\P(\tau_{\{0\}}(Y^{\L})< \tau_{\{b\}}(Y^{\L})~|~Y^{\L}_0=0).\een
 Formula \eref{eq:htegfe} allows seeing that the weight of a cycle $(a, a+b, a+b-1,\cdots,a+1,a)$ for $Y^\U$ is the same as the weight of the cycle $(a, a+1,\cdots,a+b, a)$ for $Y^\L$ (which allows proving \eref{eq:retf}, using combinatorial techniques).
  \end{rem}

\subsection{General presentation of \slt{} transition matrices using descent kernels}

\label{sec:cata}
A slight change of point of view on \slt{} transition matrices will allow us to search more efficiently the form of their time-reversal when they exist (see \Cref{pro:frsgr}), to design many \slt{} transition matrices $\L$ for which it is possible to find the time-reversal (Section \ref{sec:catal}), and, finally, to revisit some known results of the literature (the so-called, catastrophe transition matrices, see Section \ref{sec:CK}). The results collected in this section are of interest for the user searching some complete families of \slt{} transition matrices for which the invariant distribution are computable (for sake of teaching, statistical purpose, or simple curiosity).
\begin{defi}
A descent kernel $\DK=\begin{bmatrix}\DK_{i,j}\end{bmatrix}_{i,j\geq 0}$ over $\mathbb{N}$  is a \underbar{lower} triangular transition matrix $\DK_{i,j}>0 \imp j\leq i$ (with non-negative coefficients, summing to one on each row).
  \end{defi}
Each \slt{} transition matrix $\L$ can be represented uniquely as a pair $(v,\DK)$ where $v=(v_a,a\geq 0)$ is a sequence of elements of the interval $[0,1]$ (in fact, $(0,1)$ in the irreducible case, except $v_0\in[0,1)$), and $\DK$ a descending transition matrix, as follows:
\ben\label{eq:sgergf1}
\bpar{ccl}
\L_{b,a } & = & v_{b}\, \DK_{b,a}, \textrm{ for }b\geq a,\\
\L_{b,b+1}  & = & 1-v_{b},~~~~b\geq 0.\epar
\een
We will say that $(v,\DK)$ is the descent representation of $\L$.
In other words: $v_b$ is seen as the probability of descent from $b$, and $\DK$ the descent kernel, conditionally on a descent. With probability $1-v_b$, there is an ascent.

In the literature, instead of descent kernel, the word ``catastrophe'' is sometimes used, but with a slightly different construction, relying instead upon a standard birth-death process, mixed with a descend kernel  (in our representation, the random walker has to choose randomly between a $+1$ step and a descent taken according to $\DK$, see \eref{eq:sgergf1}). We think that our choice, while equivalent, is more compact, and allows to better observe the algebra into play (see e.g. Pollett \& al. \cite{MR2340220}, Brockwell \& al. \cite{MR677553}, Kapodistria \& al \cite{MR3436782}, and references therein).

\subsubsection{Representation of time-reversal of \slt{} transition matrices (descent form)}

From \Cref{theo:dqti} we see that irreducible \slt{} transition matrices $\L$ having an invariant measure and those admitting a time-reversal are the same. The representation of \slt{}-transition matrices $\L$ using descent kernels will allow us to have a better point of view on the form of their possible time-reversal.
\begin{pro} \label{pro:frsgr} If the set of time-reversals of an irreducible transition matrix $\L$ with descent representation $(v,\DK)$ is not empty, then each of its element $\U$  can be represented as follows
\ben\label{eq:sgergf2}
\bpar{ccl}
\U_{a, b} &=& u_a\, \alpha_a \beta_b\, \DK_{b,a},\textrm{ for }b\geq a,\\
\U_{a,a-1} &=& 1-u_a,~~~~a\geq 0, \textrm{ with }u_0=1
\epar\een
where $[\pi, u,\alpha,\beta]$ is a 4-tuple of sequences which satisfies\\
$(i)$~ $\pi$, $u$, $\alpha$, $\beta$ are sequences of positive real numbers, except for $\beta_0$ which is $0$ iff $v_0=0$; moreover $u_0=1$ and $u_j\in(0,1)$ for $j\geq 1$,\\
$(ii)$~$\sum_{b:b\geq a} \alpha_a \beta_b\DK_{b,a}=1$ for all $a\geq 0$,\\
$(iii)$~for all $b\geq 0$, $a\leq b$,
\ben\label{eq:gsh}
u_a\alpha_a= {1}/{\pi_a},~~\beta_b=\pi_b v_b.
\een
$(iv)~ $ $\pi_a(1-v_a)=\pi_{a+1}(1-u_{a+1})$ for all $a\geq 0$.
\end{pro}
\begin{proof} First, assume that $\pi,u,\alpha,\beta$ satisfy the properties stated in the theorem. By $(i)$ and $(ii)$, $\U$ is a transition matrix.
Let us check that  $\pi_b\L_{b,a}=\pi_a \U_{a,b}$ which is sufficient to conclude (by \Cref{theo:dqti}).\\
-- First, we have $\pi_a \L_{a,a+1}=\pi_a(1-v_a)= \pi_{a+1}(1-u_{a+1})=\pi_{a+1}\U_{a+1,a}$,\\
-- and for $b\geq a$, $\pi_b \L_{a,b}= \pi_b v_b D_{b,a}=\beta_b D_{b,a}$, while $\pi_a\U_{a,b}=\pi_au_a\, \alpha_a \beta_b\, \DK_{b,a}=\pi_bv_bD_{b,a}=\pi_b \L_{b,a}$ (by \eref{eq:gsh}).~\\
These points show that in all cases $\pi_a \U_{a,b}=\pi_b\L_{b,a}$.\\
Conversely, assume that $\U$ is a time-reversal of $\L$ (with representation $(v,D)$) for some positive measure $\pi$, that is, it satisfies $\pi_b\L_{b,a}=\pi_a \U_{a,b}$. Since $\U$ is \sut{}, it can be represented as $\U_{a,a-1}=1-u_a$ and $\U_{a,b}=u_a H_{a,b}$ for a sequence $u$ and $H$ such that $H_{a,b}>0\Rightarrow b\geq a$ (an ascent kernel). The sequence $u$ must satisfy
$\pi_b\L_{b,b+1}=\pi_b(1-v_b)=\pi_{b+1}\U_{b+1,b}=\pi_{b+1}(1-u_{b+1})$, so that $(iv)$ holds.

Let us show that $\U$ can be represented as stated in the Theorem. First, we must have $\L_{b,a}=0\equi \U_{a,b}=0$.
Since for all $b>0$, the factor $v_b>0$, $\L_{b,a}=0\equi D_{b,a}=0$ and since this must be equivalent to $\U_{a,b}=0$, it is easily seen that $\U_{a,b}=u_aH_{a,b}=u_a.\DK_{b,a}.g(a,b)$ for some positive function $g(a,b)$. This way of thinking extends to $b=0$, when $v_0>0$. If $v_0=0$, then $\L_{0,0}=0$, and since $\DK_{0,0}=1$ (because $\DK$ is a descending transition matrix), to satisfy $\U_{0,0}=u_0H_{0,0}=u_0.\DK_{0,0}.g(0,0)=0$ too, we will take $g(0,0)=0$ (in fact $\beta_0=0$ will be the needed specification).
Now, for all $b\geq a$ write
\ben\label{eq:fgr}
\pi_b v_b\DK_{b,a}=\pi_a \U_{a,b} \equi \U_{a,b}=\pi_b v_b\DK_{b,a}/\pi_a.
\een
and then if $b>0$, the variables in factor to $\DK_{b,a}$ are functions of separated variables $a$ or of $b$, so that $g(a,b)=\alpha_a\beta_b$ for some sequence $\alpha$ and $\beta$. Set $\alpha_a= {1}/({u_a\pi_a}),~~\beta_b=\pi_b v_b$, and for this choice, $\U_{a,b}=u_a\alpha_a\beta_b \DK_{b,a}$, so that $(iii)$ and $(i)$ hold. It remains to check $(ii)$. Since $\U$ is the time-reversal of $\L$, we get $\sum_{b} \pi_a \U_{a,b}=1$ and then $\sum_{b:b\geq a} \pi_a\U_{a,b}=\pi_au_a$ which implies that
$\sum_{b:b\geq a} \pi_au_a\alpha_a\beta_b \DK_{b,a}=\pi_a u_a$ and then  $\sum_{b:b\geq a} \alpha_a\beta_b \DK_{b,a}=1$.
\end{proof}

\subsubsection{Catalytic inversion of \slt{} transition matrices}

\label{sec:catal}
In this section, we introduce a tool allowing one to design many \slt{}-transition matrices $\L$ with a computable invariant measure (and computable time-reversal transition matrices $\U$). The weakness of this approach is that it is far more efficient when, instead of fixing a given $\L$ in terms of its descent representation $(v,\DK)$ only $(v_0,\DK)$ is fixed. By this method to find a complete descend kernel $(v,\DK)$ with a computable invariant measure amounts to finding a positive sequence $X$ satisfying some inequalities:
\begin{defi}\label{defi:push} Given a pair $(v_0,\DK)$ where $v_0\in[0,1)$ and $\DK$ a descent kernel. A sequence $X=(X_a,a\geq 0)$ is said to be $(v_0,\DK)$ pushable iff the two following conditions are satisfied:\\
  \bls $X_0=v_0$ and, for all $i\geq 1$, $X_i>0$ (so that $X_0=0$ is possible). \\
  \bls For all $a\geq 0$, $Y_a:=\sum_{b\geq a } X_b\DK_{b,a}$ is finite. For short, we will write $Y=X.\DK$.\\
  \bls For all $a\geq 1$, \ben\label{eq:srgr} \sum_{i=1}^a (Y_{i}/Y_0-X_{i})>0.
  \een
  \end{defi}

  \begin{theo}\label{theo:machin} Let $v_0\in[0,1)$, $\DK$ be a descent kernel, and $X$ be a  $(v_0,\DK)$ pushable sequence. Set $Y=X.\DK$, and
    \ben\label{eq:sfqdf}
    v_{a+1}=\frac{X_{a+1}}{Y_{a+1}/Y_0 +X_a(1/v_a-1)} \textrm{ for all }a\geq 0
    \een
(if $v_0=0$, take $v_1= \frac{X_1}{Y_0}(1-\frac{Y_1}{Y_0+Y_1})=\frac{X_1}{Y_0+Y_1}$ instead).
We have
\ben\label{eq:cd1}
v_a\in(0,1), \textrm{ for all }a>0,
\een
since this is equivalent to \eref{eq:srgr}. Define the  4-tuple $[u,\alpha,\beta,\pi]$ by 
  $\pi_0=Y_0$,   $u_0=1$ and for $a\geq 0$,
  \ben\label{eq:fqsdf}
  \bpar{ccl}\label{eq:uY} u_{a+1}&:=&\dis\frac{Y_{a+1}u_a}{Y_a(1-v_a)+Y_{a+1}u_a},\\
  \pi_{a+1}&:=&
  \dis\pi_a\frac{1-v_a}{1-u_{a+1}},\\
  \beta_a&:=&X_a,\\
  \alpha_a&:=&1/(u_a\pi_a)\epar
  \een
  then for this 4-tuple, the \slt{} $\L$ with representation $(v,\DK)$ has time-reversal $\U$ as defined in \Cref{pro:frsgr}, and then, both transition matrices $\L$ and $\U$ have $\pi$ as invariant measure.
\end{theo}

\begin{proof} Let us first say why \eref{eq:cd1} is equivalent to \eref{eq:srgr}. Set
  $\gamma_{a}= X_a (1/v_a-1)$.  From \eref{eq:sfqdf}, one gets that
  \ben\label{eq:gg} \gamma_{a+1}= \gamma_a+ Y_{a+1}/Y_0-X_{a+1},~~a\geq 0\een
and then $\gamma_{a+1}=\gamma_0+\sum_{i=1}^{a+1}(Y_i/Y_0-X_i)$. Since $v_a\in(0,1)\equi \gamma_a>0$, we get the result (note that $\gamma_1=Y_0+Y_1-X_1$ when $v_0=0$, so that \eref{eq:gg} holds for $\gamma_0=(Y_0+Y_1-X_1)-Y_1/Y_0+X_1=Y_0+Y_1-Y_1/Y_0$ in this case).

  It suffices to check that \Cref{pro:frsgr} applies for the tuple of sequences $[\pi,u,\alpha,\beta]$ as defined in \eref{eq:fqsdf}. 
  The condition $(i)$ is immediate, since we took $\beta_0=X_0=v_0$; the fact that $u_j\in(0,1)$ for all $j>0$ is clear.\\
  For $(ii)$ observe that the first equation of the system \eref{eq:uY} is equivalent to
\ben
\frac{1-v_a}{1-u_{a+1}}=\frac{Y_{a+1}/u_{a+1}}{Y_a/u_a}\een
so that $(Y_a/u_a)$ is proportional to $(\pi_a, a\geq 0)$ as defined in the third equation of the system \eref{eq:fqsdf}, and since $Y_0/u_0=\pi_0$ these sequences are equal. 
Write
  \[\sum_{b:b\geq a} \alpha_a\beta_b \DK_{b,a}=\frac{1}{u_a\pi_a}\sum_{b:b\geq a} X_b \DK_{b,a}=\frac{Y_a}{u_a\pi_a}=1. \]
 To obtain condition $(iii)$. Since $u_a\alpha_a= {1}/{\pi_a}$, we only need to prove that $\beta_b= X_b=\pi_b v_b$. Since $\pi_0=1$, $X_0=v_0$, the formula is true for $b=0$; let us assume that it holds for $b\leq a$, for some $a$, and let us establish that $X_{a+1}=\pi_{a+1}v_{a+1}$. From \eref{eq:sfqdf},
\be
\frac{X_{a+1}}{v_{a+1}}&=&\frac{u_{a+1}Y_{a+1}}{u_{a+1}Y_0} +X_a\l(\frac{1}{v_a}-1\r) = u_{a+1}\pi_{a+1}+\pi_a(1-v_a)=\pi_{a+1}
\ee by system \eref{eq:fqsdf}, second equation. [The case where $a=0$ and $v_0=0$ has to be treated separately: in this case $u_1= Y_1/(Y_0+Y_1)$ and since $Y_0=\pi_0$, $v_1=\frac{X_1}{Y_0}(1-Y_1/(Y_0+Y_1))$, we have $X_1/v_1= Y_0/(1-u_1)=\pi_0(1-v_0)/(1-u_1)=\pi_1$ so that $X_1/v_1$ is indeed equal to $\pi_1$\\
Finally $(iv)$ is immediate by the second equation of the system \eref{eq:fqsdf}.
\end{proof}

\subsubsection{``Catastrophe transition matrices'': analysis of \slt{} transition matrices with same descent kernel }
\label{sec:CK}
\Cref{theo:machin} gives a reformulation for the problem of finding the invariant distributions or the time-reversal of a given transition matrix $\L$, and as such it may appear a bit useless since no methods are provided to compute the pair $(X,Y)$ which is needed to conclude. The following examples show the power of this theorem: given the descent kernel $\DK$ and some parameter $v_0\in[0,1)$, it is quite easy to find many $(v_0,\DK)$ pushable sequence $X$. Even if it is still difficult to target a given sequence $v=(v_i,i\geq 0)$, it is possible to construct many sequences $v$ for which it is possible to construct the time-reversal of $(v,\DK)$.
  This allows observing the general form of integrable systems $(v,\DK)$.
The following results are comparable with those of  Pollett \& al. \cite{MR2340220}, Brockwell \& al. \cite{MR677553}, Kapodistria \& al \cite{MR3436782} (and references therein), in which various catastrophe transition matrices are investigated (in continuous-time). In these results also, it can be observed that very specific forms of catastrophe transition matrices are needed to find the invariant distributions, or absorption probabilities: each time it is a challenge to complete all computation details.

\paragraph{Geometric catastrophe}
This is the family of \slt{} transition matrices whose $(v_0,\DK)$ representation involved, for some $p\in(0,1)$, the descent kernel
\[\DK_{b,a}=p(1-p)^{b-a} +\1_{a=0}(1-p)^{b+1}, ~~ 0\leq a \leq b.\] 
For any positive sequence $X= (X_j,j\geq 0)$ and $a\geq 0$,
\be
Y_a&=&\sum_{b \geq a} X_b \DK_{b,a} 
=  p\sum_{x \geq 0} X_{a+x}(1-p)^x +\1_{a=0}\sum_{b\geq 0} X_b(1-p)^{b+1},
\ee
so that this  close formula can be effectively computed for many sequences $X$ (the $X_i's$ are the coefficients of a power series). It remains to extract the pushable sequences (those that satisfy $\sum_{i=1}^a (Y_{i}/Y_0-X_{i})>0$, for non-negative parameters and sequences $(v_0,X)$ with $X_0=v_0$, see \Cref{defi:push}).
From this, the complete description of the vectors $v$ and $u$ can be obtained as explained in \Cref{theo:machin}.

\paragraph{Binomial catastrophe}

The descent kernel $\DK$, in this case, is defined as follows
\[\DK_{b,a}=\binom{b}a p^{a} (1-p)^{b-a}, ~~ 0\leq a \leq b.\]
For $X=(e^{-\lambda}\lambda^b/b!,b\geq0)$ the  Poisson distribution $(P^{(\lambda)}_b,b\geq 0)$ with parameter $\lambda$, the corresponding $Y$ is Poisson distributed with parameter $\lambda p $, i.e $Y_a=P^{(\lambda p)}_a$. In this case $Y_i/Y_0=(p\lambda)^i/i!$ and then $\sum_{i=1}^a (Y_{i}/Y_0-X_{i})=\sum_{i=1}^a ((p\lambda)^i/i! - e^{-\lambda}\lambda^i/i!)=\sum_{i=1}^a \frac{\lambda^ie^{-\lambda}}{i!}(p^ie^\lambda-1) $ is indeed positive for $p$ such that $\lambda p\geq 1-e^{\lambda}$ at least, since $\sum_{i=1}^a \frac{\lambda^ie^{-\lambda}}{i!}(p^ie^\lambda)\geq \lambda p$ (the value taken for $a=1$) and  for each $a\geq 1$, $ \sum_{i=1}^a \frac{\lambda^ie^{-\lambda}}{i!}(-1)\geq\sum_{i=1}^{+\infty} \frac{\lambda^ie^{-\lambda}}{i!}(-1)=-(1-e^{-\lambda})$ ,  so that $(v_0,X)$ is pushable when $\lambda p\geq 1-e^{-\lambda}$.\\
For $v_0=X_0$, with $v_0\in(0,1)$, then 
\[v_{a+1}= \frac{e^{-\lambda}\lambda^{a+1}/(a+1)!}{e^{-\lambda p}(p\lambda)^{a+1}/(a+1)!/e^{-\lambda p}+e^{-\lambda}\lambda^a/a!(1/v_a-1)}=\frac{1}{ e^\lambda p^{a+1}+(1/v_a-1)(a+1)/\lambda}.\]
From this formula the sequence $(v_a,a\geq 0)$ is characterized.
From here, set $u_0=1$, and compute successively, for $a\geq 0$,
\[u_{a+1}= \frac{e^{-\lambda p}(\lambda p)^{a+1}/(a+1)! u_a}{e^{-\lambda p}(p\lambda)^{a}/a!(1-v_a)+e^{-\lambda p}(\lambda p)^{a+1}/(a+1)! u_a}=\frac{(\lambda p) u_a}{(1-v_a)(a+1)+(\lambda p) u_a}.\]
After that, the values of $\pi$ can be obtained from the third formula in the system \eqref{eq:uY}.

\paragraph{Uniform catastrophe}
The case of uniform catastrophe is described by the following descent kernel $\DK_{b,a}=1/(b+1)$ for $a\in\cro{0,b}$. The case where $X_b=(b+1) \rho_b$ for $\rho_b$ a probability measure with full support on $\mathbb{N}$ and a finite mean,  is integrable. Denote $\overbar{\rho}_b=\sum_{k\geq b}\rho_k$ (the tail distribution function). The computation of $Y$ gives $Y_a= \overbar{\rho}_a$ so that $Y_0=1$. The pushability condition is $\sum_{i=1}^a (\overbar{\rho}_i - (i+1)\rho_i)\geq 0$. 
For $v_0=X_0$, with $v_0\in(0,1)$, compute successively the $v_a$ using : 
\[v_{a+1}= \frac{(a+2) \rho_{a+1}}{\overbar{\rho}_{a+1} + (a+1) \rho_a}.\]
Set $u_0=1$, and compute the $u_a$ using:
\[u_{a+1}= \frac{\overbar{\rho}_{a+1}u_a}{\overbar{\rho}_{a}(1-v_a)+\overbar{\rho}_{a+1}u_a},\]
After that, the values of $\pi$ can be obtained from the third formula in the system \eqref{eq:uY}.

\section{Particular models}

\subsection{Back to the tridiagonal case}
\label{sec:sqshtqdf}\label{sec:cr}

\paragraph{About the formulas for invariant distributions.} In the tridiagonal  case, by \eref{eq:sgfqd}, the invariant distribution $(\pi_a,a\geq 0)$ associated with the tridiagonal transition matrix $\TTT$ is unique (since it is  \sut), and it is proportional to $(p_a^{(t)}:=\prod_{j=1}^{a} \frac{\TTT_{j-1,j}}{\TTT_{j,j-1}},a\geq 0)$, and to $(p_a^{(\U)}:=\frac{\det( \Id-\TTT_{[0,a-1]})}{\prod_{j=1}^{a}\TTT_{j,j-1}},a\geq 0)$ in Theorem \ref{theo:InvDistU} for the \sut{} case. Therefore, in the tridiagonal case, these two formulas must coincide.\par
The fact that $p_a^{(t)}=p_a^{(\U)}$ is a consequence of \Cref{theo:InvDistU} and  Remark \ref{rem:shh}$(ii)$ (and it is also a consequence of \eref{eq:tjt} and \Cref{pro:rytou}).\\
In the \slt{} case, it is a bit more complex since \Cref{theo:f978sdqs} only deals with finite \slt{} transition matrices and we need then to use \Cref{pro:ruygfs}. We then take $\L^{(n)}$ as described in this \Cref{pro:ruygfs}.
In the tridiagonal case, $\L^{(n)}_{i,j}=\TTT_{i,j}$ for all $i,j\leq n$, except for $\L^{(n)}_{n,n}=\TTT_{n,n}+\TTT_{n,n+1}$. 
For any $a<n$, we have, by Theorem \ref{theo:f978sdqs}
\[\rho_a^{(n)} = c_n \det( \Id-\L^{(n)}_{[a+1,n]})\prod_{i=1}^a \TTT_{i-1,i}.\]
The determinant,  by the matrix tree theorem, coincides with the weight of trees rooted at $a$, on the graph with vertex set $[a,n]$, and edge set $\{(i,j)~: \TTT_{i,j}>0,i,j\leq n\}$. Since the only decreasing edges are the $(j,j-1)$, there is a single tree on $[a,n]$ rooted at $a$, it is the tree with edges $\{(j,j-1), j \in \cro{a+1,n}\}$.
Hence
\[\rho_a^{(n)} = c_n \l(\prod_{i=a+1}^n \TTT_{i,i-1}\r)\l(\prod_{i=1}^a \TTT_{i-1,i}\r)\1_{a\leq n}=c'_n \prod_{i=1}^a \frac{\TTT_{i-1,i}}{\TTT_{i,i-1}}\1_{a\leq n}.\]
Set, as done in \Cref{pro:ruygfs},
\[\eta_a^{(n)}=\rho_a^{(n)}/\rho^{(n)}_0=\1_{a\leq n}\prod_{i=1}^a \frac{\TTT_{i-1,i}}{\TTT_{i,i-1}}.\]
It remains to prove that the three conditions $(a)$, $(b)$ and $(c)$ of the proposition are satisfied.\\
$(a)$ We need to take $S_a= \prod_{i=1}^a \frac{\TTT_{i-1,i}}{\TTT_{i,i-1}}$ (which is $p_a^{(t)}=p_a^{(\U)}$ by the way). For $b$ fixed, we have $\sum_{a:a\geq b+1}S_a\TTT_{a,b}=S_{b+1}\TTT_{b+1,b}$ and this is indeed $<+\infty$.\\
$(b)$ Let $\eta_a:=S_a=p_a^{(t)}=p_a^{(\U)}$. The convergence of $\eta_a^{(n)}\xrightarrow[n\to \infty]{}{} \eta_a$ is obvious.\\
$(c)$ For $a$ fixed, and $n>a+1$, we have $\sum_{j\geq n}\TTT_{j,a}=0$, so that $\eta^{(n)}_n \sum_{j\geq n}\TTT_{j,a}\to 0$ as $n\to\infty$, as required.

\paragraph{Recurrence criterion}
In the tridiagonal case, the criterion for recurrence is known to be \eref{eq:dqsdt} (that is $\sum_{k\geq 0}\prod_{j=1}^k \frac{\TTT_{j,j-1}}{\TTT_{j,j+1}}=+\infty$).
Let us show that it is equivalent to \Cref{theo:dfdqdq} in this case $\Big(\lim_{b=+\infty} u_1(b)=1$ with $u_1(b)=\M_{1,0} \frac{\det(\Id -\TTT_{[2,b-1]})}{\det(\Id-\TTT_{[1,b-1]})}\Big)$.
If $\TTT$ is tridiagonal
\[\det(\Id-\TTT_{[1,b-1]})=(1-\TTT_{1,1}) \det(\Id -\TTT_{[2,b-1]})-\TTT_{1,2}\TTT_{2,1} \det(\Id -\TTT_{[3,b-1]})\] and writing $D_{a,b}:=\det(\Id-\TTT_{[a,b-1]})$, we have more generally
\ben\label{eq:thhu}
D_{i,b-1}=(1-\TTT_{i,i}) D_{i+1,b-1}-\TTT_{i,i+1}\TTT_{i+1,i} D_{i+2,b-1},~~\textrm{ for }i+1\leq b-1.\een
Set 
\[Z_{i,b-1}=\frac{D_{i,b-1}}{D_{i+1,b-1}\,\TTT_{i,i-1}}\]
so that $u_1(b)=1/Z_{1,b-1}$.
Formula \eref{eq:thhu} rewrites
\ben\label{eq:relZ}
Z_{i,b-1} &= c_i+{a_{i+1}}\,/\,{Z_{i+1,b-1}} \textrm{ for } i\leq b-2
\een
for 
\ben \label{eq:qab}
q_i := \frac{\TTT_{i,i+1}}{\TTT_{i,i-1}},~~~a_{i+1}=-q_i,~~~c_i= 1+q_i.
\een
Notice that ``the last'' term, for $i=b-1$,
\[Z_{b-1,b-1}=\frac{D_{b-1,b-1}}{D_{b,b-1}\,\TTT_{b-1,b-2}}=\frac{1-\TTT_{b-1,b-1}}{\TTT_{b-1,b-2}}=\frac{\TTT_{b-1,b-2}+\TTT_{b-1,b}}{\TTT_{b-1,b-2}}=1+{q_{b-1}}=c_{b-1}\] so
that \eref{eq:relZ} can be used to express $Z_{1,b-1}$ in terms of the $a_i$'s and $c_i$'s as follows: if one sees the relation \eref{eq:relZ} as a kind of continued fraction expansion, what we want to do, is to produce a formula for the so-called convergents:
\ben\label{eq:1b-1h}
Z_{1,b-1}:=c_1+\cfrac{a_2}{c_2+\cfrac{a_3}{\ddots+\cfrac{\ddots}{c_{b-2}+\cfrac{a_{b-1}}{c_{b-1}}}}};\een
One can compute the value of the ``finite continuous fraction'' expressed in \eref{eq:1b-1h}: set  \[A_{0}=1, A_{1}=c_1=1+q_1,B_{0}=0,B_1=1\] and computing successively
\ben\left\{\begin{array}{ccl}
             A_k&=&c_kA_{k-1}+a_kA_{k-2}=(1+q_k)A_{k-1}-q_{k-1}A_{k-2},\\
             B_k&=&c_kB_{k-1}+a_kB_{k-2}=(1+q_k)B_{k-1}-q_{k-1}B_{k-2}
\end{array}\right.
\een
for $k$ from 2 to $b-1$, we get
\[Z_{1,b-1}= A_{b-1}\,/\,B_{b-1}.\]
We can proceed to the computation of $A_{b-1}$ and $B_{b-1}$, by observing that for any $k\leq b-1$, since $A_k=(1+q_k)A_{k-1}-q_{k-1}A_{k-2}$ and $B_k=(1+q_k)B_{k-1}-q_{k-1}B_{k-2}$, we then have for $k\leq b-1$~:
\be \alpha_k:=A_k -q_k A_{k-1} &=&A_{k-1}-q_{k-1}A_{k-2}=\alpha_{k-1},\\
\beta_k :=B_k -q_{k}B_{k-1} &=&B_{k-1}-q_{k-1}B_{k-2}=\beta_{k-1}.
\ee
From what we see that $(A_k -q_k A_{k-1},k<h)$ and $(B_k -q_{k}B_{k-1},k<h)$ are both constant, but are subjected to different initial conditions. For  $C_k=A_k$ or  $C_k=B_k$, with $c$ encoding the initial condition,
\ben
C_k&=&q_kC_{k-1}+c \imp C_k= q_k(q_{k-1}C_{k-2}+c)+c
= C_0.\prod_{j=1}^kq_j+c\sum_{j=1}^k  \prod_{i=j+1}^k q_i.
\een
One gets for $k\leq b-1$, setting $Q_k:=\prod_{j=1}^kq_j=Q_{k-1}q_k$, $F_k:=\sum_{j=1}^k  \l(\prod_{i=1}^j q_i\r)^{-1}$
\be
A_k&=& 
Q_k(1+F_k) ,
~~ B_k=Q_kF_k.
\ee
Hence,
\[Z_{1,b-1}= 1+1/F_{b-1}\]
and since $F_k:=\sum_{j=1}^k  \l(\prod_{i=1}^j q_i\r)^{-1}=\sum_{j=1}^k  \prod_{i=1}^j \frac{\TTT_{i,i-1}}{\TTT_{i,i+1}}$ 
and then one observes that the convergence of $Z_{1,b-1}$ to 1 is indeed equivalent to Karlin \& McGregor criterion \eref{eq:dqsdt}.

\paragraph{Positive recurrence criterion}
The positive recurrent criterion is $\sum_{k\geq 0}\prod_{j=1}^k \frac{\TTT_{j-1,j}}{\TTT_{j,j-1}}<\infty$, while it is $\sum_{a=1}^\infty \frac{\det\left( \Id-\U_{[0,a-1]}\right)}{\prod_{j=1}^{a}\U_{j,j-1}}<\infty$ in the $\U$ case. Section \ref{sec:sqshtqdf} explained why, in the tridiagonal case,  the formulae obtained in the \sut{} and tridiagonal cases coincides.

  \subsection{Non-uniqueness of invariant measures in the \slt{} case, and of main right eigenvectors in the \sut{} case: martingale connection}
\label{sec:teheesgs}

Let us first state a simple fact:
\begin{pro}\label{pro:qffthqeh}Let $(X_k,k \geq 0)$ be a Markov chain with irreducible \sut{} transition matrix $\U$, and let ${\cal F}=\l({\cal F}_n,n\geq 0\r)=(\sigma(X_i,i\leq n), n\geq 0)$ be the filtration generated by the $X_i$'s. Consider a function $R:\mathbb{N}\to [0,+\infty)$ (which we want to view also as a vector $R=\begin{bmatrix}R_{0},&R_1,&R_2,&\cdots\end{bmatrix}^{t}$). The process $(R_{X_k},k\geq 0)$ is ${\cal F}$-martingale iff $R$ is a \underline{right eigenvector} of $\U$ associated with the eigenvalue 1.
\end{pro}
\begin{proof} Since $(X_k,k\geq 0)$ is a Markov chain, $(R_{X_k},k\geq 0)$ is a martingale iff
$\E ( R_{X_{k+1}} | {\cal F}_k) = \E ( R_{X_{k+1}} | X_k) =R_{X_k}$. Since $(X_k,k\geq 0)$ is time homogeneous and discrete, observe that $(R_{X_k},k\geq 0)$ is a martingale iff, for all $i\geq 0$,
$E (R _{Z_1} | Z_0=i )= R_i$, which is equivalent to
$\sum_{j \geq i-1} \U_{i,j} R_j=R_i$ as claimed.
  \end{proof}
  Since $\U$ is a transition matrix, $R=\begin{bmatrix} 1, &1, &1 \cdots\end{bmatrix}$ is a right eigenvector. It is also a consequence of this Proposition since $\begin{bmatrix}R_{X_i},i\geq 0\end{bmatrix}=\begin{bmatrix} 1, &1, &1\cdots\end{bmatrix}$ defines a martingale.
\begin{proof}[\textbf{Proof of  \Cref{pro:qfqeh2}}]
 In view of \Cref{pro:qffthqeh} and the invariance of the constant eigenvector $R=\begin{bmatrix} 1, &1, &1 \cdots\end{bmatrix}$ it suffices to prove the existence of a vector $R$ with positive non constant coordinates and an irreducible \sut{}  transition matrix $\U$ such that $\U R=R$.

 For that purpose, it suffices to take a positive sequence $(R_i,i\geq 0)$ which does not reach its infimum nor its supremum; for example, consider the bounded sequence $(R_k,k\geq 0)$ defined by
\[R_k = 1_{k\,{\sf mod }\, 2 =1}\l(1+\frac{1}{10 k}\r) + 1_{k\,{\sf mod}\, 2 = 0} \l(2-\frac1{10(k+1)}\r),~~\textrm{ for }k\geq 0.\]
For such a sequence, for any fixed $k$, $R_k$ is neither the maximum nor the minimum in the set $\{R_{j},j\geq k-1, j\neq k\}$: under this condition, $R_j$ is a barycentre of the others $(R_k, k \in\{j-1\} \cup[k+1,+\infty))$: there exists some non-negative $(\alpha_{j,\ell},\ell\geq j-1)$ summing up to 1, with $\alpha_{j,j-1}>0$,  and at least one $\alpha_{j,\ell}>0$ for $\ell >j$, such that
\[R_j= \sum_{\ell\geq j-1} \alpha_{j,\ell} R_{\ell}.\]
Taking $\U_{i,j}=\alpha_{i,j}$ for all $i,j$, one sees that $\U$ has right eigenvector $R$, and $\U$ is irreducible.
\end{proof}    

\begin{proof}[\textbf{Proof of \Cref{theo:hetjyfd}$(c)$} ]
Consider $\U$ a \sut{} irreducible transition matrix satisfying $(i)$, and assume that additionally to $R^{(0)}=[R^{(0)}_i,i\geq 0]$ with $R_i^{(0)}:=1$, there is a second bounded positive right eigenvector $R^{(1)}=:[R^{(1)}_i,i\geq 0]$: for $k \ in\{0,1\}$
\ben\label{eq:UR} \sum_{b} U_{a,b}R^{(k)}_{b} = R^{(k)}_a,~~a\geq 0.\een Since $\U$ is a \sut{} transition matrix it admits a unique invariant distribution $\pi$ (by \Cref{theo:InvDistU}). Now, set
for all $a,b\geq0$, $\L_{b,a}= \pi_a U_{a,b}/\pi_{b}$ the time-reversal of $\U$. 
Formula \eref{eq:UR} implies that
\[ \sum_{b}\pi_b R^{(k)}_{b}\,\L_{b,a}  = \pi_aR^{(k)}_a,~~a\geq 0\]
in other words, $\L$ possesses $\l[ \pi_aR^{(k)}_a,a\geq 0 \r]$ as invariant measures for $k=0$ and $k=1$, and these two measures are not proportional.
\end{proof}

\subsection{A class of integrable \lt{} transition matrices}
\label{eq:fqrfdqe} In Section \ref{sec:catal}, we described a strategy to design some \slt{} transition matrices $\L$ for which the time-reversal $\U$ can be computed (for some invariant measure $\pi$ computed simultaneously). Here, we present a large family of transition matrices for which the invariant distribution can be computed directly.\par
Denote by $\se_j=(\1_{i=j},i\geq 0)$ the vector with a 1 in entry $j$ only (the first vector is $\se_0$).
\begin{defi}\label{def:fqgrzd}
  A \slt{}-transition matrix is said to be Col(0)-triangular if for any $c\geq 1$,
\ben\label{eq:0simple}
\L_{\bullet,c}=\alpha_c \L_{\bullet,0}+ \sum_{\ell=0}^c a_{\ell,c}\, \se_\ell,
\een
where, $\L_{\bullet,c}$ denotes the column $c$ of $\L$, the first column being $\L_{\bullet,0}$.
 \end{defi}
The first column $\L_{\bullet,0}$ is general, and for any $c$,  $\L_{\bullet,c}$ is essentially proportional to $\L_{\bullet,0}$, except its $c+1$ first entries, indexed from 0 to $c$. Since $\L_{0,c}=\cdots=\L_{c-2,c}=0$, for $c\geq 1$, only the entries $\L_{c,c}$ and $\L_{c-1,c}$ are free. 

We claim that it is possible to solve the system $\eta = \eta \L$ when $\L$ is a \slt{}, irreducible and Col(0)-triangular. The idea is to keep on hold the first equation until the end of the resolution:
\ben\label{eq:qfg}
\eta_0 = \eta L_{\bullet,0}=\sum_k \eta_k L_{k,0}.
\een
Suppose that $\eta$ is the solution of \eref{eq:qfg} and $\eta =\eta \L$. For $c \geq 1$, we have
\ben\label{eq:qdthed1}
\eta_c=\eta \L_{\bullet,c}=\alpha_c\, \eta\,\L_{\bullet,0}+ \eta \sum_{\ell=0}^c a_{\ell,c}\, \se_\ell
= \alpha_c\, \eta_0+\sum_{\ell=0}^c a_{\ell,c}\, \eta_\ell
\een
and it is then apparent that $(\eta_c,c\geq 1)$ is solution of a standard triangular linear system in which $\eta_0$ is seen as a parameter. It remains to check if the obtained solution of \eref{eq:qdthed1} solves \eref{eq:qfg} or not (which corresponds to the case where a solution exists or none, respectively).

\noindent\textbf{Example:} Consider a sequence of vectors $V^{\geq k}=\begin{bmatrix} V_0 1_{0\geq k}\\V_1 1_{1\geq k}\\\vdots\end{bmatrix}$ (indexed by $k$) where, vertically, the entries are non increasing:
$1>V_0>V_1>V_2>\cdots > 0$.
The vectors $V^{\geq k}$ are essentially proportional, and $V:=V^{\geq 0}$ has the sequence $(V_i,i\geq 0)$ as entries.
The simplest case of Col(0)-triangular transition matrices are those in which the columns $\L_{\bullet,c}$  are essentially proportional. Consider $\alpha_0=1,\alpha_1,\alpha_2,\dots$ and the matrix $\L$ such that
\ben\label{eq:form}
\L_{\bullet,c}&=& \alpha_c V^{\geq c-1}= \alpha_c \Big[L_{\bullet,0} - \sum_{j=0}^{c-2} V_j\se_j\Big] \textrm{ for all } c\geq 0.
\een
Since  $\L$ is a \slt-transition matrix, the condition $\sum_{c} \L_{r,c}=1$ for all $r$ becomes:
\ben\label{eq:rhjjf}\l( \alpha_0+\cdots \alpha_{r+1}\r)V_{r}&=&1 ~~\textrm{ for }r\geq 0,
\een
which implies $\sum_{\ell=0}^{r+1} \alpha_\ell=1/V_r$, and since $1+\alpha_1=\alpha_0+\alpha_1= \frac{1}{V_0}$, we obtain
\ben\label{eq:fqqf}
\alpha_1= \frac{1}{V_0}-1, \textrm{ for }r>0, \alpha_{r+1}=\frac1{V_r}-\frac{1}{V_{r-1}}.\een
Since $(V_i,i\geq 0)$ is decreasing, it is easily seen that $r\mapsto \l(\sum_{i=1}^r \alpha_i\r)$ is increasing so that  all the $\alpha_i$ are positive. 
We have
\ben\label{eq:ffgdq}
\eta_k = \eta \L_{\bullet,k}=\eta_0\alpha_k(1-V_0\cdots-V_{k-2}), ~~~ k \geq 0\een
we then need that $\sum_{i=0}^{+\infty} V_i\leq 1$ and
\[\eta_0=\sum_k \eta_k L_{k,0}=\sum_k \eta_k V_k<+\infty\]
which is equivalent to
\ben\label{eq:rhjjf2}
1&=&\sum_k \alpha_k V_k(1-V_0\cdots-V_{k-2}).
\een
Using \eref{eq:fqqf}, it can be written as 
\ben\label{eq:VV}
V_0+ V_1\l(\frac{1}{V_0}-1\r)+\sum_{k\geq 2} \l(\frac{1}{V_{k-1}}-\frac{1}{V_{k-2}}\r) V_k(1-V_0\cdots-V_{k-2})&=&1.
\een 
If  $\L$ satisfies \eref{eq:form} and \eref{eq:VV}, then its invariant measures $\eta$ can be computed thanks to \eref{eq:ffgdq}.

\subsection{Some \slt{} and \sut{} transition matrices associated with BDP Markov chains}
\label{sec:hit}

Let $\TTT$ be an irreducible tridiagonal transition matrix. As explained at the beginning of Section \ref{sec:Intro}, the invariant measure of $\TTT$, and criterion of recurrence and positive recurrence are available.
 Now define two associated transition matrices $\U$ and $\L$ as follows.
 Let $X$ be a Markov process with transition matrix $\TTT$. The increasing steps of $X$ are $+1$ while decreasing steps are $-1$.
 Now, define
 \ben
 \bpar{ccl}
\U_{a,b}&=&\P\l(X_\tau^{\downarrow}=b ~|~X_0=a\r),\\ 
\L_{a,b}&=&\P\l(X_\tau^{\uparrow}=b ~|~X_0=a\r)\epar
\een
where
\be
\tau^{\downarrow}&=&\inf\{t>0~:~X_{t}=X_{t-1}-1\},\\
\tau^{\uparrow} &=&\inf\{t>0~:~X_{t}=X_{t-1}+1\}.
\ee
In words, if one observes $X$ only at the times $(t_j,j\in \Z)$ following a decreasing step, then the sequence of observations is a $\U$-Markov process on $\{0,1,2,3,\cdots\}$. If one observes $X$ only at the times $(t_j,j\in \Z)$ following an increasing step is a $\L$-Markov process on $\{1,2,3,\cdots\}$.

The transition matrices $\U$ and $\L$ can be computed using path decompositions:
Set $\Loop_i^{\TTT}:=\l(1-\TTT_{i,i}\r)^{-1}$.
For all $i\geq 0$, we have $\U_{i,i-1}= \Loop_i^{\TTT}\, \TTT_{i,i-1}$, and for $j\geq i$,
\[ \U_{i,j}= \l[\prod_{b=i}^{j} \Loop_b^{\TTT}\,\TTT_{b,b+1}\r]\,\Loop_{j+1}^{\TTT}\, \TTT_{j+1,j}.\]
We have, for $i\geq 1$, $\L_{i,i+1}= \Loop^{\TTT}_i\, \TTT_{i,i+1}$, for $1\leq j\leq i$,
\[\L_{i,j}=\l[\prod_{b=j}^i \Loop^{\TTT}_b \TTT_{b,b-1}\r]\, \Loop_{j-1}^{\TTT}\, \TTT_{j-1,j}.\]

\begin{rem} The matrices $\U$ (respectively $\L$) are transition matrices when $\TTT$ is recurrent, since, in this case, with probability 1, starting from any position $i$, a Markov chain $X$ with transition matrix $\TTT$, will have some decreasing steps (resp. increasing steps) which ensure that, for any $i$, $\sum_j \U_{i,j}=1$ (resp. $\sum_j \L_{i,j}=1$). It is however possible for a BDP to have globally a.s. a finite number of steps $-1$ in the transient case, so that $\U$ is not always a transition matrix. Since starting from any point $i$, a Markov chain with transition matrix $\TTT$ will have a $+1$ step with probability 1 (by irreducibility), $\L$ is well defined even when $\TTT$ is transient.
\end{rem}

Set 
\ben\label{eq:muUL}
\bpar{ccl}\pi^{\U}_a&=&\pi^{\TTT}_{a+1}\,\TTT_{a+1,a},~~~a\geq 0,\\
\pi^{\L}_a&=&\pi^{\TTT}_{a-1}\,\TTT_{a-1,a},~~~ a\geq 1\epar.
\een
\begin{pro}\label{pro:drhqd}
  Assume that $\TTT$ is irreducible.
  \bir
  \itr  $\TTT$ is recurrent (resp. positive recurrent) iff $\L$ is irreducible and recurrent (resp. positive recurrent).  If $\L$ is a transition matrix then, $\L$ admits $\pi^\L$ as invariant measure (in all cases, including $\L$  transient).
  \itr If $\TTT$ is recurrent and $\U$ is a transition matrix, then $\U$ is irreducible and recurrent. $\TTT$ is positive recurrent iff $\U$ is positive recurrent. 
  The measure $\pi^\U$ is invariant by $\U$ (in all cases, including $\U$  transient, and even, when $\U$ is not a transition matrix!).
\eir
\end{pro}
\begin{proof} In the proof, we treat simultaneously $(i)$ and $(ii)$.
First, the fact that the recurrence of $\TTT$ is equivalent to irreducibility and recurrence of $\U$ (resp. of $\L$) is clear. By irreducibility of $\TTT$, the recurrence of $\TTT$ implies that each edge $(a,a+1)$ and $(a,a-1)$ are traversed infinitely often by a Markov chain with transition matrix $\TTT$ so that $\U$ and $\L$ are recurrent. The converse use the same type of argument. \par
  If $\TTT$ is positive recurrent, then by the ergodic theorem, the proportion of time passed at $a$ by a Markov chain with transition matrix $\TTT$ converges to $\pi^\TTT_{a}$, and then, the proportion of time passed at an increasing step $(a,a+1)$ is $\pi^\TTT_a \TTT_{a,a+1}$,  and  the proportion of time passed at a decreasing step $(a,a-1)$ is $\pi^\TTT_a \TTT_{a,a-1}$. A simple consequence of that is that \eref{eq:muUL} holds in the positive recurrent case (since $\sum_{a} \pi_{a+1}^\TTT$ converges, $\sum_{a} \pi_{a+1}^\TTT \TTT_{a+1,a}$  and $\sum_{a} \pi_{a+1}^\TTT \TTT_{a,a+1}$ converges too).\par
Now, let us check the statements concerning the invariant measures.  

By \eref{eq:sgfqd}, 
\[
  \pi_a^\U = \pi_{a+1}^\TTT \TTT_{a+1,a} = \frac{\prod_{j=1}^{a+1}\TTT_{j-1,j}}{\prod_{j=1}^{a} \TTT_{j,j-1}}.
  \]
  We want to prove that $\pi^\U$ is invariant by $\U$. 
\begin{align*}
\pi_a^\U\U_{a,b} 
&= \frac{\TTT_{0,1} \dots \TTT_{a,a+1}}{\TTT_{1,0}\dots \TTT_{a,a-1}} \TTT_{b+1,b} \l(\prod_{i=a}^{b} \TTT_{i,i+1}\r)\l( \prod_{i=a}^{b+1} \Loop_i^\TTT\r)\\
&= \frac{\TTT_{0,1} \dots \TTT_{b,b+1}}{\TTT_{1,0}\dots \TTT_{b,b-1}}\l[ \TTT_{a,a+1}\Loop_a^\TTT \l(\prod_{i=a+1}^{b+1}\TTT_{i,i-1} \Loop_i^\TTT\r)\r]= \pi_b^\U \L_{b+1,a+1}
\end{align*}
since the last bracket is $\L_{b+1,a+1}$; this allows to conclude to the invariance of $\pi^\U$ by $\U$:
\[
	\sum_{a:a\leq b+1} \pi_a^\U\U_{a,b} = \pi_b^\U \sum_{a:a\leq b+1}\L_{b+1,a+1}=\pi_b^\U \sum_{j:j\leq b+2}\L_{b+1,j}=\pi_b^\U.  
  \]
  Now, we want to prove that $\pi^\L$ is invariant by $\L$. Write 
  $\pi_a^\L = \pi_{a-1}^\TTT \TTT_{a-1,a} = \frac{\prod_{j=1}^{a-1}\TTT_{j-1,j}}{\prod_{j=1}^{a-1} \TTT_{j,j-1}} \TTT_{a-1,a} =\frac{\prod_{j=1}^{a}\TTT_{j-1,j}}{\prod_{j=1}^{a-1} \TTT_{j,j-1}}$, so that
  
\begin{align*}
\pi_a^\L\L_{a,b} 
  &=\frac{\prod_{j=1}^{a}\TTT_{j-1,j}}{\prod_{j=1}^{a-1} \TTT_{j,j-1}}
\l[\prod_{k=b}^a \Loop^{\TTT}_k \TTT_{k,k-1}\r]\, \Loop_{b-1}^{\TTT}\, \TTT_{b-1,b}  \\
&=\l[\frac{\prod_{j=1}^{b-1}\TTT_{j-1,j}}{\prod_{j=1}^{b-1} \TTT_{j,j-1}}\TTT_{b-1,b} \r] \M_{a,a-1}
\l[\prod_{k=b}^a \TTT_{k-1,k}\Loop^{\TTT}_k \r]\, \Loop_{b-1}^{\TTT}\,\\
  &=\pi^\L_{b-1}
\l[\prod_{k=b-1}^{a-1} \Loop^{\TTT}_k \TTT_{k,k+1} \r] \l(\TTT_{a,a-1}\Loop^{\TTT}_a\r)=\pi^\L_{b-1} \U_{b-1,a-1}. 
\end{align*}
From here, $\sum_{a} \pi_a^{\L}\L_{a,b}= \sum_{a} \pi_{b-1}^\L \U_{b-1},a=\pi^\L_{b-1}$ if $\U$ is a transition matrix.
  \end{proof}

\begin{rem} This example is one of the simplest  \sut{} and \slt{} transition matrices one can construct using stopping times of a BD Markov chain. One can construct many other transition matrices by designing other stopping times: for example, for a \sut{} transition matrix: starting at $k$, starts the trajectory when it hits $k-1$, or at the first time where the 5 last steps are $(+1,+0,+1,-1,+1)$.  
  \end{rem}

  \subsection{Repair shop Markov chain}
  \label{sec:RPMC}
 
  The following is a well known Markov chain that is present in different textbooks. It is an integrable system, where recurrence, positive recurrence and transience have been characterized. Most of the results about this chain can be found in \cite{B13} under the tag {\it repair shop}. Nevertheless, the methods that we will use to obtain the same results are based on \Cref{theo:InvDistU} and are therefore of different nature.\par
     The repair shop chain is defined as the Markov chain $X_n$ given by:
  \[
  	X_{n+1} = (X_n-1)_+ +Z_{n+1}
  \]
  where $(Z_n,n\geq 0)$ is a sequence of i.i.d. random variables with distribution $(a_k,k\geq 0)$, meaning that $\P(Z_n = k)= a_k$ for every $k\geq 0$.
  This chain models the number of broken machines in a repair shop, where each day one broken machine is repaired (when there is at least one available to repair), and where the number of new machines that need to be repaired day $n+1$ is $Z_{n+1}$.  The transition matrix ${\sf A}$ associated with this chain is
 \[{\sf A}=\begin{bmatrix}
     a_0 & a_1 & a_2 & a_3 & \cdots \\
     a_0 & a_1 & a_2 & a_3 & \cdots \\
     0   & a_0 & a_1 & a_2 &  \cdots \\
     0   & 0   & a_0 & a_1 &  \cdots \\
     \vdots &  \vdots   &   \vdots &  \vdots &  \vdots 
   \end{bmatrix}.\] 
\begin{note} Some generalizations of the repair shop Markov chain appear in the literature, notably, in relation with queueing theory; see e.g. Abolnikov \& Dukhovny \cite{MR1136432} and references therein.
\end{note}
 
\subsubsection{Positive recurrence criterion}

Set $m:=\sum_s sa_s$ the mean of $Z_1$, i.e. of the distribution $(a_0,a_1,\cdots)$. 
\begin{pro}The transition matrix ${\sf A}$ is positive recurrent if and only if $m<1$.
\end{pro}
\begin{proof} By Lemma \ref{lem:etyyu}, equation \eqref{det1}
\ben
C_N:=\frac{\det({\sf Id}-\A_{[0,N]})}{\prod_{j=1}^{N+1}\A_{j,j-1}}&=& \sum_{s\in S^N} \left(\prod_{j=1}^{\ell(s)}(\Id-\A)_{s_{j-1}+1,s_j} \right)\left(\prod_{j\in[0,N-1] \setminus s} \A_{j+1,j}\right)a_0^{-N-1}\\
&=& \sum_{s\in S^N} \prod_{j=1}^{\ell(s)} \frac{(\Id-\A)_{s_{j-1}+1,s_j} }{a_0}.
\een
Since $(\Id-\A)_{x,y}=1_{x=y}-a_{y-x+1-1_{x=0}}$, we have
\be
C_N&=&\sum_{s\in S^N}\prod_{j=1}^{\ell(s)} \frac{1_{s_j-s_{j-1}=1}-a_{s_j-(s_{j-1}+1)+1-1_{s_{j-1}=0}}}{a_0}.
\ee
This is a kind of product of transitions that measures the passage from $s_{j-1}$ to $s_j$.
For the increments except the first one, define
\[t_\delta =\frac{1_{\delta=1}-a_{\delta}}{a_0},~~\delta\geq 1\]
and for the first increment, define
\[t^f_{\delta}= \frac{(\Id-\A)_{0,\delta}}{a_0}= \frac{1_{\delta =0}-a_\delta}{a_0}, \delta \geq 0.\]
Consider the generating functions which encode the increments of the sequence $(s_j,j\geq 0)$
\be
G_f(x)&=&\sum_{\delta\geq 0}t^f_\delta x^\delta= \frac{1-a_0}{a_0}-\sum_{s>0} \frac{a_{s}}{a_0}x^s,\\
G(x)&=&\sum_{\delta\geq 1}t_\delta x^\delta=  \frac{1-a_1}{a_0}x-\sum_{s>1} \frac{a_{s}}{a_0}x^s.
\ee
Now, $C_N=[x^{N}]G_f(x)/(1-G(x))$ (notation for the extraction of the coefficient of $x^N$ in the generating function $G_f(x)/(1-G(x))$) so that
\[\sum_{N\geq 0} C_N=  \lim_{x\to 1} G_f(x)/(1-G(x));\]
the sum of the coefficients of a series is obtained by a simple evaluation at 1 but only when the power series converge at this point. Since  $G^f(1)=0$ and $1-G(1)=0$, we apply the L'hôpital rule, which says
\[\sum_{N\geq 0} C_N= \lim_{x\to 1} \frac{G_f(x)}{1-G(x)}= \lim_{x\to 1}\frac{ G'_f(x)}{-G'(x)}.\]
We have,
\be
G'_f(1)&=&-\sum_{s>0} \frac{a_s s}{a_0} =-\frac{m}{a_0},\\
-G'(1) &=& -\l[(1-a_1)-\sum_{s>1} sa_s\r]/a_0=-\frac{1-m}{a_0};
\ee
therefore $\sum_{N\geq 0} C_N= {m}/({1-m})$,      
from what we see that this converges iff $m<1$, which is then the sufficient and necessary condition for positive recurrence.
\end{proof}

\subsubsection{Recurrence criterion}

\begin{pro}The transition matrix ${\sf A}$ is positive recurrent if and only if $m\leq 1$.
\end{pro}
\begin{proof}
Let $\alpha_N:=\det((\Id-\J)_{[0,N]})$ for 
\[\J = \begin{bmatrix}
a_1 & a_2 & a_3 & a_4 & \cdots\\
a_0 & a_1 & a_2 & a_3 & \cdots \\
0   & a_0 & a_1 & a_2 & \cdots \\
0   & 0   & a_0 & a_1 & \cdots \\
\vdots   &  \vdots &  \vdots &  \vdots &  \vdots 
 \end{bmatrix}.\]
where $\A$ is the matrix presented in the previous section.
Since $(\Id-\A)_{[2,N]} = (\Id-\A)_{[1,N-1]}= (\Id-\J)_{[0,N-2]}$, by \Cref{theo:dfdqdq}
\begin{align}\label{equiv}
\A \textrm{~is~recurrent~} \equi \l(\lim_{N\rightarrow \infty}a_0\frac{\alpha_{N-1}}{\alpha_N} = 1\r).
\end{align}
By Lemma \ref{lem:etyyu}
\ben
\alpha_N &=& \sum_{s\in S^N} \Big(\prod_{j=1}^{\ell(s)}(\Id-\J)_{s_{j-1}+1,s_j} \Big)\Big(\prod_{j\in[0,N-1] \setminus s} \J_{j+1,j}\Big)\\
&=& \sum_{s\in S^N} \prod_{j=1}^{\ell(s)} (\Id-\J)_{s_{j-1}+1,s_j}a_0^{s_j-s_{j-1}-1}. 
\een
Since $(\Id-\J)_{x,y}=1_{x,y}-a_{y-x+1}$, we have
\be
\alpha_N&=&\sum_{s\in S^N}\prod_{j=1}^{\ell(s)} \left(1_{s_j-s_{j-1}=1}-a_{s_j-(s_{j-1}+1)+1}\right)a_0^{s_{j}-s_{j-1}-1}.
\ee
Again, this is a kind of product of transitions that weight the passage from $s_{j-1}$ to $s_j$.
We set
\[t_\delta =\left(1_{\delta=1}-a_{\delta}\right)a_0^{\delta-1},~~\delta\geq 1\]
Consider the generating function which encodes the increments of the sequence $(s_j,j\geq 0)$
\be
G_\J(x)&=&\sum_{\delta\geq 1}t_\delta x^\delta=  \frac{(1-a_1)}{a_0}(a_0x)-\sum_{s>1} \frac{a_{s}}{a_0}\left(a_0x\right)^s.
\ee
Now, $\alpha_N = [x^N]\l({1}/({1-G_\J(x))}\right)$.
    By \Cref{theo:dfdqdq}, $a_0\frac{\alpha_{N-1}}{\alpha_N}$ is non decreasing in  $N$ and bounded by 1, so that  $\lim_Na_0\frac{\alpha_{N-1}}{\alpha_N}$ exists, and better than that it is equal to $\sup_N a_0\frac{\alpha_{N-1}}{\alpha_N}=S\leq 1$. A standard result of calculus (the ratio test) applies:  
    the radius $r$ of convergence of $1/(1-G_\J)$ satisfies 
    \[1/r =a_0/S.\]
This allows to complete a step in our reasoning:
   \[
    \A \textrm{~is~recurrent~}\equi \l(\lim_Na_0\frac{\alpha_{N-1}}{\alpha_N}= 1\r)\equi \l(r= 1/a_0\r).\]
Now notice that the function $1/(1-z)$ has a radius of convergence 1 and since $G_\J$ has radius of convergence at least $1/a_0$, the function $1/(1-G_\J(x))$ has radius of convergence given by the
\be
\overline{R}&:=_{(a)}&(1/a_0) \wedge \inf\{|x|,  G_\J(x)=1\}\\
&=_{(b)}&(1/a_0) \wedge \inf\{x>0,  G_\J(x)=1\}\\
&=_{(c)}& \inf\{x>0,  G_\J(x)=1\}
\ee
Equality $(a)$ follows the preceding discussion, equality $(b)$ follows the fact that the coefficients of $G_\J$ are non positive (from $a_2/a_s$), and equality $(c)$ holds because $G_\J(1/a_0)=1$.
We infer that
\begin{align}\label{equiv23}
\A\textrm{~is~recurrent~} \equi \overline{R}  = 1/a_0.
\end{align}
To finish this sequence of equivalence, it is enough to prove that
\[\overline{R} = 1/a_0 \equi m\leq 1.\]
Start by noting that $G_\J(1/a_0) = (1/a_0)(1-\sum_{s\geq 1}a_s) = 1$.
Now $G_\J'(x) = 1-\sum_{s\geq 1}sa_s(a_0x)^{s-1}$ and therefore $G_\J'(1/a_0) = 1-m$.

If $m>1$, then the function $G_\J$  is locally decreasing at $1/a_0$. Hence, there exists some $\varepsilon>0$ such that $G_\J(1/a_0-\varepsilon)> 1$, which implies together with $G_\J(0)=0$ and the intermediate value theorem, that $G_\J=1$ as at least one solution on $[0,1/a)$ which implies that  $\overline{R}<1/a_0$. (transient).

If $m\leq 1$, then $0\leq \sum_{s\geq 1}sa_s(a_0x)^{s-1} \leq \sum_{s\geq 1}sa_s=m$ when $0\leq x \leq 1/a_0$ so that $G'_\J(x)\geq 1-m\geq 0$ on $[0,1/a_0]$ (in fact, $G'_\J>0$ on $[0,1/a_0)$). Since $G_\J(1/a_0)=1$ and $G_\J(0)=0$, and $G_\J$ is increasing monotone on $[0,1/a_0)$, then $\overline{R}=1/a_0$.
\end{proof}

\section{Appendix}

\subsection{continuous-time counterparts}
\label{sec:CTC}
A continuous Markov process $X=(X_t,t\geq 0)$ is a continuous-time process described by means of a generator $\GG = (\GG_{i,j}: i,j\in S)$, where $S$ ($\N$ for us) is the state space, and which satisfies $\GG_{i,i} = -\sum_{j\neq i}\GG_{i,j}$ (each of these sums being finite), so that each row of $\GG$ sums up to zero. The value $\GG_{i,j}$ is a rate (for $i\neq j$), and can be seen as the parameter of an exponential distribution: it is the jump rate for the process, when its value is $i$, at which it jumps at $j\neq i$. For more information on this type of process see \cite{N98,P08}. An invariant measure $\pi^C$ for the continuous Markov process chain is a non-negative measure satisfying $\pi^C \GG = 0$.

The jump process associated with $X$ is the discrete-time Markov chain $Y= (Y_k,k\geq 0)$, defined by
  \[Y_k=X(\tau_k),~~\textrm{ for } k\geq 0\]
  where $\tau_0=0$, and for $k\geq 1$, $\tau_k=\inf\{t:t>\tau_{k-1}, X_t\neq X_{\tau_{k-1}}\}$, that is the $k$th jump time of $X$. The transition matrix of $Y$ is $\M=(\M_{i,j}:i,j\in \mathbb{N})$ defined as
\[
	\M_{i,j} = -\GG_{i,j}/\GG_{i,i}, \quad \forall i\neq j \qquad \text{ and }\qquad \M_{i,i}=0 \quad \forall i\in \N.
\]
The properties of positive recurrence, null recurrence and transience are inherited from the jump chain to the continuous chain under non-explosion assumptions (Theorem 3.4.1. and Theorem 3.5.3 in \cite{N98}). Also the knowledge of the (or an) invariant measure of one of these processes (either of $Y$ or of $X$) allows one to deduce the corresponding invariant measure of the other, by using:
\[
	\pi_i = -\pi_i^{C} \GG_{i,i}, \textrm{ for all } i \geq 0.
\]
To finish it is important to notice that: \sut{} (\slt{}) $\GG$ iff that $\M$ is \sut{} (\slt{}).
For this reason, our results apply to continuous Markov processes with \slt{} and \sut{} generator matrices $\GG$.

\subsection{BDP and orthogonal polynomials}
\label{sec:ortho-pol}

Karlin \& McGregor approach relies on the study of the spectral properties of the tridiagonal transition matrix $\TTT$ (see notation in \eref{eq:MM}) and its connection with a family of orthogonal polynomials $(Q_i, i\geq 0)$ defined as follows: set $Q_0(x)=1$, $p_0Q_1(x)=x-r_0$, and
\be
xQ_j(x)&=&q_j Q_{j-1}(x)+r_jQ_j(x)+p_jQ_{j+1}(x),~j\geq 1;
\ee
and this can be rewritten in the following form
\ben\label{eq:grht} Q(x):=\begin{bmatrix} Q_0(x) & Q_1(x)&Q_2(x)&\cdots \end{bmatrix}^{t},~~~xQ(x)= \TTT Q(x).\een 
 Observe that $Q(x)$ is then an eigenvector of $\TTT$ associated with the eigenvalue $x$.

Karlin \& McGregor \cite{KmG} prove that there exists a unique measure $\psi$ on $[-1,1]$ for which the family $(Q_i, i\geq 0)$ forms an orthogonal family (more precisely $\pi_j\int_{-1}^1 Q_i(x)Q_j(x)d\psi(x)=\1_{i=j}$, where $\pi$ is the invariant measure of $\TTT$), and further
\ben(\TTT^{n})_{i,j}=\pi_j\int_{-1}^1 x^n Q_i(x)Q_j(x)\,d\,\psi(x).\een

Their approach is somehow more natural in the continuous settings: define the transition rate matrix
\ben
{\sf G}=
\begin{bmatrix}
  -(\lambda_0+\mu_0)  & \lambda_0  & 0  & 0 & 0 & \cdots \\
  \mu_1 & -(\lambda_1+\mu_1)   & \lambda_1  & 0 & 0 & \cdots \\
  0  & \mu_2 & -(\lambda_2+\mu_2)   & \lambda_2  & 0 & \cdots \\
  0  & 0 & \mu_3 & -(\lambda_3+\mu_3)   & \lambda_3  & \cdots  \\
  \vdots  & \vdots & \ddots & \ddots & \ddots & \ddots 
\end{bmatrix},
\een and a second family of orthogonal polynomials $(\tilde{Q}_i, i\geq 0)$ as follows:
\ben\label{eq:pol}
\bpar{ccl}-x\tilde{Q}_0(x)&=& -(\lambda_0+\mu_0)\tilde{Q}_0(x)+\lambda_0 \tilde{Q}_1(x),\\
-x\tilde{Q}_n(x)&=& \mu_n\tilde{Q}_{n-1}(x)-(\lambda_n+\mu_n)\tilde{Q}_n(x)+\lambda_{n}\tilde{Q}_{n+1}(x),\\
\tilde{Q}_0(x)&\equiv&1
\epar\een
so that \eqref{eq:pol} can be rewritten as
\ben\label{eq:grht2} -x\tilde Q(x)= {\sf G} \tilde Q(x).\een  
These polynomials are the orthogonal polynomials of a solvable Stieljes moment problem associated with a regular probability measure $\Psi$ on $[0,+\infty)$.  There exists a unique measure $\Psi$ on $[0,+\infty)$ for which the $(\tilde Q_i, i\geq 0)$ forms an orthogonal family (\cite{KmG}), more precisely, $\pi_j^C\int_{-1}^1 \tilde Q_i(x)\tilde Q_j(x)d\Psi(x)=\1_{i=j}$, where $\pi^C$ is the explicitly known invariant measure of the continuous-time process.

Set $P'(t)= P(t)\GG$,
for $t\geq 0$ and $P(0)=\Id$. The matrix $P(t)$ is the transition matrix of the continuous-time BD process, and $P_{i,j}(t)$ is   the probability that the state of the chain is $j$ at time $t$ given that it started at time 0 in state $i$ (a detail: $\mu_0$ is not assumed to be 0, the case were absorption at 0 may occur is included). Then Karlin \& McGregor defined ``formally''
\ben\label{eq:djrjdg}
f_i(x,t)=\sum_{j\geq 0} P_{i,j}(t)\tilde Q_j(x)\een
and in vectorial notation
\[f(x,t)= P(t)\tilde Q(x)\imp \partial f(x,t)/\partial t= P'(t)\tilde Q(x)=P(t)\GG\tilde Q(x)=-xf(x,t),\]
subject to the initial condition $f(x,0)=\tilde Q(x)$. From here $f(x,t)=\exp(-xt)\tilde Q(x)$, so that
\[f_i(x,t)=\exp(-xt)\tilde Q_i(x).\]
Now, reinterpret $f_i(x,t)$ on the $\L^2$ space in which we are working in, equipped with its basis of orthogonal polynomials $(\tilde Q_j(x),j\geq 0)$. The extraction of $P_{i,j}(t)$ in \eref{eq:djrjdg} can be done using the orthogonality of the $\tilde Q_j$'s,
\[\P_{i,j}(t)= \l[\int f_i(x,t) \tilde Q_j(x) d\Psi(x)\r]/\l[ \int \tilde Q_j(x)^2 d\Psi(x)\r] \]
(they set $ \int \tilde Q_j(x)^2 d\Psi(x) = 1/{\pi^C}_j$).

Main point in the construction: the orthogonality of the polynomials, means, since $\tilde Q_0(x)=1$, and $\int \tilde Q_0d\Psi=1$, that $\int \tilde Q_jd\Psi=0$ for $j\geq 1$, the moments $\int x_n d\Psi$ can be expressed in the $\tilde Q_n$ (since $\tilde Q_n$ has degree $n$)

Karlin \& McGregor constructed their study by establishing a correspondence
 between the set of matrices $\GG$ of continuous-time BD processes, and the set of solvable Stieltjes moment problem. From here, the measure $\Psi$ encodes somehow in an indirect way the polynomials $(Q_i,i\geq 0)$ (as a transform), which are solution to $-x\tilde Q(x)=\GG\tilde Q(x)$ and then they encode the spectral properties of $\GG$, which drives the behaviour of $P(t)$ (by \eref{eq:djrjdg}). The extraction of recurrence criterion from here (see \cite[p.370 - 376]{KmG2}) is done by expressing the recurrence in terms of a certain property of $\Psi$, which in turns, are shown to be expressible in terms of the coefficients of $\GG$ (which provides the criteria given at the beginning of Section \ref{sec:Intro}, in the discrete-time version of the BD process). 

The methods developed by Karlin \& McGregor are really elegant and satisfying from a theoretical point of view. These methods connect probability theory, algebra, measure theory (specifically the ``moment problem'') and the theory of orthogonal polynomials. However, the focus made on the map $\GG\mapsto \Psi$, which is something that can be compared with the recourse of Fourier transform in other fields of probability theory, has the effect to lock a bit these studies, and this, for two reasons. The first one is that the correspondence is exact, so that, these tools are not simply available for any extension of the class of BD processes. Secondly, the measures $\Psi$ are in general not known, nor computable, so that, $\Psi$ is used as a formal encoding tool, rather that as a computing tool that helps to make some computation : there is just a handful of important cases in which it can be computed (see e.g. Schoutens \cite{MR1761401}). The approach we propose is different since it is centred on direct computations of quantity of interests. Limitations exist, but they are not the same at all. The criterion of recurrence/transience we provide, do not rely on the moments computation of any measure.

\small
\bibliographystyle{abbrv}

\bibliography{mybib}

\newpage
\tableofcontents

\end{document}

%% file: preambuleB.tex
\usepackage[utf8]{inputenc}
\usepackage{amsmath,indentfirst,qsymbols,graphicx,psfrag,amsfonts,stmaryrd,hyperref,color,enumerate,amsthm,dsfont}
\usepackage{url}
\usepackage[dvipsnames]{xcolor}
\usepackage{cleveref}
\usepackage{xparse}
\usepackage{todonotes}
\usepackage{comment}

\usepackage[english]{babel}
\usepackage{mdframed,ulem}

\newcommand{\W}{{\sf Weight}}

\usepackage{caption}
\usepackage{subcaption}

\newcommand{\ra}[1]{\overleftarrow{#1}}
\newcommand{\overbar}[1]{\overline{#1}}

\usepackage{ascii}

\def \A#1#2{{\sf Trees}_{#2}(#1)}

\def \E{\mathbb{E}}

\newtheorem{lem}{Lemma}
\newtheorem{defi}[lem]{Definition}
\newtheorem{pro}[lem]{Proposition}
\newtheorem{theo}[lem]{Theorem}

\newtheorem{note}[lem]{Note}

\newtheorem{rem}[lem]{Remark}

\numberwithin{equation}{section}
\numberwithin{lem}{section}

\def \bls{{\tiny $\blacksquare$ \normalsize }}

\def \1{\textbf{1}}

\def \Z{\mathbb{Z}}

\def \ar#1{\overrightarrow{#1}}

\def \TTT{{\mathsf{T}}}
\def \A{{\mathsf{A}}}
\def \J{{\mathsf{J}}}
\def \GG{{\mathsf{G}}}

\def \bpar#1{\left\{\begin{array}{#1} }
\def \epar { \end{array}\right.}

\def \N{\mathbb{N}}

\def \ba{\begin{align}}
\def \ea{\end{align}}
\def \be{\begin{eqnarray*}}
\def \ee{\end{eqnarray*}}
\def \ben{\begin{eqnarray}}
\def \een{\end{eqnarray}}

\def \beq{\begin{equation}}
\def \eq{\end{equation}}

\def \build#1#2#3{\mathrel{\mathop{\kern 0pt#1}\limits_{#2}^{#3}}}

\def \ba{{\bf a}}

\def \dis{\displaystyle}

\def \equi{\Leftrightarrow}
\def \eref#1{(\ref{#1})}

\def \P{{\mathbb{P}}}

\def \imp{\Rightarrow}

\def \l{\left}

\def \r{\right}
\def \sous#1#2{\mathrel{\mathop{\kern 0pt#1}\limits_{#2}}}
\def \sur#1#2{\mathrel{\mathop{\kern 0pt#1}\limits^{#2}}}

\newqsymbol{`B}{\mathcal{B}}
\newqsymbol{`E}{\mathbb{E}}
\newqsymbol{`I}{\mathbb{I}}
\newqsymbol{`N}{\mathbb{N}}
\newqsymbol{`O}{\Omega}
\newqsymbol{`P}{\mathbb{P}}
\newqsymbol{`Q}{\mathbb{Q}}
\newqsymbol{`R}{\mathbb{R}}
\newqsymbol{`W}{\mathbb{W}}
\newqsymbol{`Z}{\mathbb{Z}}
\newqsymbol{`a}{\alpha}
\newqsymbol{`e}{\varepsilon}
\newqsymbol{`o}{\omega}
\newqsymbol{`t}{\tau}
\newqsymbol{`w}{{\cal W}}
\def\cro#1{[#1]}

\newcommand{\compact}{ \topsep0pt   \itemsep=0pt   \partopsep=0pt   \parsep=0pt}

\newcounter{c}
\def \bir{\begin{itemize}\compact \setcounter{c}{0}}
\def \itr{\addtocounter{c}{1}\item[($\roman{c}$)]} 
\def \eir{\end{itemize}\vspace{-2em}~}

\newcounter{d}
\def \bia{\begin{itemize}\compact \setcounter{d}{0}}
\def \eia{\end{itemize}\vspace{-2em}~}

\newcounter{b}
\def \bi{\begin{itemize}\compact \setcounter{b}{0}}

\def \ei{\end{itemize}\vspace{-2em}~}